\documentclass[11pt,english]{article}
\usepackage[T1]{fontenc}
\usepackage[latin9]{inputenc}
\usepackage{geometry}
\usepackage{color}
\usepackage{babel}
\usepackage{array}
\usepackage{units}
\usepackage{amsmath}
\usepackage{amsthm}
\usepackage{amssymb}
\usepackage{stmaryrd}
\usepackage[all]{xy}
\usepackage[unicode=true,pdfusetitle,
 bookmarks=true,bookmarksnumbered=false,bookmarksopen=false,
 breaklinks=false,pdfborder={0 0 0},pdfborderstyle={},backref=false,colorlinks=true]
 {hyperref}

\makeatletter

\providecommand{\tabularnewline}{\\}

\theoremstyle{plain}
\newtheorem{thm}{\protect\theoremname}[section]
  \theoremstyle{plain}
  \newtheorem{conjecture}[thm]{\protect\conjecturename}
  \theoremstyle{plain}
  \newtheorem{cor}[thm]{\protect\corollaryname}
  \theoremstyle{remark}
  \newtheorem{rem}[thm]{\protect\remarkname}
  \theoremstyle{definition}
  \newtheorem{defn}[thm]{\protect\definitionname}
  \theoremstyle{plain}
  \newtheorem{lem}[thm]{\protect\lemmaname}
  \theoremstyle{plain}
  \newtheorem{prop}[thm]{\protect\propositionname}
  \theoremstyle{remark}
  \newtheorem{claim}[thm]{\protect\claimname}

\usepackage{appendix}
\usepackage[all]{xy}

\usepackage{marginnote}
\newtheorem{notation}[thm]{Notation}

\makeatother

  \providecommand{\claimname}{Claim}
  \providecommand{\conjecturename}{Conjecture}
  \providecommand{\corollaryname}{Corollary}
  \providecommand{\definitionname}{Definition}
  \providecommand{\lemmaname}{Lemma}
  \providecommand{\propositionname}{Proposition}
  \providecommand{\remarkname}{Remark}
\providecommand{\theoremname}{Theorem}

\begin{document}
\global\long\def\rot{\mathrm{rot} }
\global\long\def\rotm{\overline{\mathrm{rot}} }
\global\long\def\alt{\mathrm{Alt} }
\global\long\def\p{{\cal P} }
\global\long\def\F{\mathrm{F} }
\global\long\def\Tr{\mathrm{Tr} }
\global\long\def\tr{\mathrm{tr} }
\global\long\def\soc{\mathrm{Soc} }

\title{The Markoff Group of Transformations\\
in Prime and Composite Moduli}

\author{Chen Meiri and Doron Puder\\
with an Appendix by Dan Carmon}
\maketitle
\begin{abstract}
The Markoff group of transformations is a group $\Gamma$ of affine
integral morphisms, which is known to act transitively on the set
of all positive integer solutions to the equation $x^{2}+y^{2}+z^{2}=xyz$.
The fundamental strong approximation conjecture for the Markoff equation
states that for every prime $p$, the group $\Gamma$ acts transitively
on the set $X^{*}\left(p\right)$ of non-zero solutions to the same
equation over $\nicefrac{\mathbb{Z}}{p\mathbb{Z}}$. Recently, Bourgain,
Gamburd and Sarnak proved this conjecture for all primes outside a
small exceptional set.

In the current paper, we study a group of permutations obtained by
the action of $\Gamma$ on $X^{*}\left(p\right)$, and show that for
most primes, it is the full symmetric or alternating group. We use
this result to deduce that $\Gamma$ acts transitively also on the
set of non-zero solutions in a big class of composite moduli. 

Our result is also related to a well-known theorem of Gilman and Evans,
stating that for any finite non-abelian simple group $G$ and $r\ge3$,
the group $\mathrm{Aut}\left(\F_{r}\right)$ acts on at least one
``$T_{r}$-system'' of $G$ as the alternating or symmetric group.
In this language, our main result translates to that for most primes
$p$, the group $\mathrm{Aut}\left(\F_{2}\right)$ acts on a particular
$T_{2}$-system of $\mathrm{PSL}\left(2,p\right)$ as the alternating
or symmetric group.
\end{abstract}
\tableofcontents{}

\section{Introduction\label{sec:Introduction}}

The Markoff surface $\mathbb{X}$ is the affine surface in $\mathbb{A}^{3}$
defined by the equation\footnote{Sometimes the Markoff equation is written as $x^{2}+y^{2}+z^{2}=3xyz$.
However, these two equations are equivalent in the sense that their
integer solutions are related bijectively by $\left(x,y,z\right)\longleftrightarrow\left(3x,3y,3z\right)$.
This bijection holds also for solutions in $\nicefrac{\mathbb{Z}}{p\mathbb{Z}}$
for every prime $p\neq3$.}
\begin{equation}
x^{2}+y^{2}+z^{2}=xyz.\label{eq:markoff}
\end{equation}
The Markoff triples ${\cal M}$\marginpar{${\cal M}$} is the set
of positive integer solutions to Equation (\ref{eq:markoff}), such
as $\left(3,3,3\right)$. The Markoff group of automorphisms of $\mathbb{X}$
is the group $\Gamma$\marginpar{$\Gamma$} generated by permutations
of the coordinates and the Vieta involutions $R_{1}$, $R_{2}$ and
$R_{3}$\marginpar{$R_{i}$} where $R_{3}\left(x,y,z\right)=\left(x,y,xy-z\right)$
and $R_{1}$ and $R_{2}$ are defined analogously. It is easy to see
that ${\cal M}$ is invariant under $\Gamma$ and Markoff proved that
$\Gamma$ acts transitively on ${\cal M}$ \cite{markoff1879formes,markoff1880formes}.
Let $\Delta$ be the group generated by $\Gamma$ and the involutions
that replace two of the coordinates by their negatives. Then the set
$\mathbb{X}\left(\mathbb{Z}\right)$ of integer solutions to (\ref{eq:markoff})
has two $\Delta$-orbits: $\left\{ \left(0,0,0\right)\right\} $ and
its complement $X^{*}\left(\mathbb{Z}\right)\stackrel{\mathrm{def}}{=}\mathbb{X}\left(\mathbb{Z}\right)\setminus\left\{ \left(0,0,0\right)\right\} $. 

\subsection*{Prime Moduli}

If $p$ is a prime number, then $\mathbb{X}\left(\nicefrac{\mathbb{\mathbb{Z}}}{p\mathbb{Z}}\right)$
is the finite set of solutions to (\ref{eq:markoff}) in $\nicefrac{\mathbb{Z}}{p\mathbb{Z}}$,
and we denote $X^{*}\left(p\right)=\mathbb{X}\left(\nicefrac{\mathbb{Z}}{p\mathbb{Z}}\right)\setminus\left\{ \left(0,0,0\right)\right\} $\marginpar{$X^{*}\left(p\right)$}.
The strong approximation conjecture for the Markoff equation (\ref{eq:markoff})
states that for every prime $p$, the reduction mod $p$ of the set
of Markoff triples ${\cal M}\to X^{*}\left(p\right)$ is onto. This
is clearly equivalent to $\Gamma$ acting transitively on $X^{*}\left(p\right)$.
Recently, Bourgain, Gamburd and Sarnak proved this conjecture for
all primes outside of a small exceptional set:
\begin{thm}[Bourgain-Gamburd-Sarnak \cite{BGS-I2016markoff}]
\label{thm:BGS}Let $E$ be the set of primes for which $\Gamma$
does not act transitively on $X^{*}\left(p\right)$. For any $\varepsilon>0$,
the number of primes $p\le T$ with $p\in E$ is at most $T^{\varepsilon}$,
for $T$ large enough.

Moreover, for any $\varepsilon>0$, the largest $\Gamma$-orbit in
$X^{*}\left(p\right)$ is of size at least $\left|X^{*}\left(p\right)\right|-p^{\varepsilon}$,
for $p$ large enough (whereas $\left|X^{*}\left(p\right)\right|\sim p^{2}$).
\end{thm}
Let $\Gamma_{p}$\marginpar{$\Gamma_{p}$} be the finite permutation
group induced by the action of $\Gamma$ on $X^{*}\left(p\right)$.
In the current work we study the nature of this group. The first step
here is to notice that $\Gamma_{p}$ preserves a block structure as
follows: 

For $\left(x,y,z\right)\in X^{*}\left(p\right)$ denote by $\left[x,y,z\right]$\marginpar{$\left[x,y,z\right]$}
the block of all solutions obtained from $\left(x,y,z\right)$ by
sign changes, so 
\[
\left[x,y,z\right]\stackrel{\mathrm{def}}{=}\left\{ \left(x,y,z\right),\left(x,-y,-z\right),\left(-x,y,-z\right),\left(-x,-y,z\right)\right\} .
\]
Then $\Gamma_{p}$ preserves this block structure. Let $Y^{*}\left(p\right)$\marginpar{$Y^{*}\left(p\right)$ }
denote the set of blocks in $X^{*}\left(p\right)$, and $Q_{p}$\marginpar{$Q_{p}$}
denote the permutation group induced by the action of $\Gamma$ (or
$\Gamma_{p}$) on $Y^{*}\left(p\right)$. Simulations suggest the
following conjecture:
\begin{conjecture}
\label{conj:always alt or sym}For every $p\ge5$, the permutation
group $Q_{p}$ is the full alternating or symmetric group.
\end{conjecture}
This conjecture was also raised, independently, in \cite[Conjecture 1.3]{magee2016cycle},
where the authors also state precisely for which primes one can expect
the alternating group ($p\equiv3\mod16$) and for which the full symmetric
group ($p\not\equiv3\mod16)$. If this conjecture holds, then roughly
speaking (we give the precise formulation in Theorem \ref{thm:strong approximation for square-free composite}
below), $\Gamma$ acts transitively on the solutions of (\ref{eq:markoff})
modulo $n$, for every square free.

Here we prove this conjecture for most primes. More particularly,
we prove it for every $p\equiv1\left(4\right)$ outside the exceptional
set from Theorem \ref{thm:BGS}, and for density-1 of the primes $p\equiv3\left(4\right)$: 
\begin{thm}
\label{thm:1 mod 4}If $p\equiv1\left(4\right)$ and $Q_{p}$ is transitive,
then $Q_{p}$ is the full alternating or symmetric group on $Y^{*}\left(p\right)$. 
\end{thm}
Namely, $Q_{p}$ is the full alternating or symmetric group for all
$p\equiv1\left(4\right)$ outside the exceptional set from Theorem
\ref{thm:BGS}. In fact, our proof yields that for every $p\equiv1\left(4\right)$,
the group $\Gamma$ acts as the full alternating or symmetric group
on the large component described in Theorem \ref{thm:BGS}. In the
case $p\equiv3\left(4\right)$, our proof is more involved and requires
one further assumption:
\begin{thm}
\label{thm:3 mod 4: alternating given specific conditions}Let $p$
be a prime. Assume that:

\begin{itemize}
\item $p\equiv3\left(4\right)$.
\item $Q_{p}$ is transitive.
\item The order of $\frac{3+\sqrt{5}}{2}\in\mathbb{F}_{p^{2}}$ is at least
$32\sqrt{p+1}$.
\end{itemize}
Then $Q_{p}$ is the full alternating or symmetric group on $Y^{*}\left(p\right)$.
\end{thm}
The number $\frac{3+\sqrt{5}}{2}$ is related to the special solution
$\left[3,3,3\right]\in Y^{*}\left(p\right)$: its order inside $\mathbb{F}_{p^{2}}$
gives the length of the cycle of the transformation $\left[x,y,z\right]\mapsto\left[x,z,xz-y\right]$
containing the element $\left[3,3,3\right]$. For details see Sections
\ref{sec:Preliminaries} and \ref{subsec:Primitivity-for-p=00003D3(4)}.

As shown in Appendix \ref{sec:Dan-Carmon}, the condition regarding
the order of $\frac{3+\sqrt{5}}{2}$ is satisfied for density-1 of
the primes\footnote{A set of primes ${\cal A}$ has \emph{density 1 }if $\lim_{n\to\infty}\frac{\left|{\cal A}\cap{\cal P}_{n}\right|}{\left|{\cal P}_{n}\right|}=1$,
where ${\cal P}_{n}=\left\{ 1<p\le n\,\middle|\,p~\mathrm{is~prime}\right\} $.
In fact, the set of primes for which $\frac{3+\sqrt{5}}{2}$ has order
at least $32\sqrt{p+1}$ satisfies something slightly stronger than
density 1 \textendash{} see Appendix \ref{sec:Dan-Carmon}.}, hence
\begin{cor}
\label{cor:3 mod 4 - alternating for density-1}For density-1 of all
primes $p\equiv3\left(4\right)$, the group $Q_{p}$ is the full alternating
or symmetric group on $Y^{*}\left(p\right)$.
\end{cor}

\subsection*{Composite Moduli}

Let $n$ be a positive integer which is square-free, so $n=p_{1}\cdots p_{k}$
where $p_{1},\ldots,p_{k}$ are distinct primes. Let $\mathbb{X}\left(n\right)$\marginpar{$\mathbb{X}\left(n\right)$}
denote the set of solutions to the Markoff equation (\ref{eq:markoff})
in $\nicefrac{\mathbb{Z}}{n\mathbb{Z}}$. By the Chinese Remainder
Theorem, $\mathbb{X}\left(n\right)=\mathbb{X}\left(p_{1}\right)\times\ldots\times\mathbb{X}\left(p_{k}\right)$,
and let $X^{*}\left(n\right)=X^{*}\left(p_{1}\right)\times\ldots\times X^{*}\left(p_{k}\right)$\marginpar{$X^{*}\left(n\right)$ }
be the set of solutions which are non-zero modulo any of the primes
composing $n$. The action of $\Gamma$ on $\mathbb{X}\left(n\right)$
is the diagonal action on the $\mathbb{X}\left(p_{i}\right)$, and
the subset $X^{*}\left(n\right)$ is invariant under this action.
Denote the corresponding permutation group \marginpar{$\Gamma_{n}$}$\Gamma_{n}$.
Is the action on $X^{*}\left(n\right)$ transitive? It turns out that
this would follow from Conjecture \ref{conj:always alt or sym} and
indeed holds true for the cases of that conjecture we establish:
\begin{thm}
\label{thm:strong approximation for square-free composite}Let $n=p_{1}\cdots p_{k}$
be a product of distinct primes. If for every $j=1,\ldots,k$, $Q_{p_{j}}\ge\mathrm{Alt}\left(Y^{*}\left(p_{j}\right)\right)$,
then $\Gamma$ acts transitively on $X^{*}\left(n\right)$.\\
In particular, if conjecture \ref{conj:always alt or sym} holds,
then $\Gamma$ acts transitively on $X^{*}\left(n\right)$ for every
square-free $n$.
\end{thm}
\begin{cor}
\label{cor:transitive for composite with factors we handle}Let $\p$
denote the set of primes that satisfy the assumptions of Theorem \ref{thm:1 mod 4}
or of Theorem \ref{thm:3 mod 4: alternating given specific conditions}.
Then for every set of distinct primes $p_{1},\ldots,p_{k}\in\p$,
$\Gamma$ acts transitively on $X^{*}\left(p_{1}\cdots p_{k}\right)$.
\end{cor}
Bourgain, Gamburd and Sarnak already proved Corollary \ref{cor:transitive for composite with factors we handle}
for primes $p\equiv1\left(4\right)$ for which $\Gamma_{p}$ is transitive.
This result should appear in the series announced in \cite{BGS-announcement2016markoff}.
We stress that our proof is entirely different: while Bourgain, Gamburd
and Sarnak improve their techniques from the proof of Theorem \ref{thm:BGS}
so that the argument work for several primes simultaneously, our proof
is group-theoretic and uses Theorem \ref{thm:BGS} as a black box.
Both proofs rely on solutions containing the parabolic elements $\pm2$
\textendash{} see Figure \ref{fig:line of arguments} and Section
\ref{sec:Preliminaries}. 

For $n=p_{1}\cdots p_{k}$ as above, we use the notation $Y^{*}\left(n\right)=Y^{*}\left(p_{1}\right)\times\ldots\times Y^{*}\left(p_{k}\right)$\marginpar{$Y^{*}\left(n\right)$}
for the set of blocks in $X^{*}\left(n\right)$ and $Q_{n}$\marginpar{$Q_{n}$}
for the permutation group induced by the action of $\Gamma$ on $Y^{*}\left(n\right)$.
Note that these blocks are given by sign changes modulo every prime
separately and are usually of size $4^{k}$ each (if all primes are
odd). It is quite straight-forward to prove that under the assumptions
of Theorem \ref{thm:strong approximation for square-free composite},
$\Gamma$ acts transitively on $Y^{*}\left(n\right)$, using composition
factors of $Q_{n}$. It requires some further argument to show that
$\Gamma$ acts transitively on the full set $X^{*}\left(n\right)$.
We elaborate in Section \ref{sec:Strong-Approximation-for-composite}.
\begin{rem}[Regarding the classification of finite simple groups]
 At this point we would like to remark on the dependence of our results
on the Classification of Finite Simple Groups (CFSG)\marginpar{CFSG}.
We use the classification only in the proof of Theorem \ref{thm:3 mod 4: alternating given specific conditions}:
we first give an elementary proof that for a prime $p$ satisfying
the assumptions in the theorem, $Q_{p}$ is a primitive permutation
group\footnote{Recall that a permutation group $G\le\mathrm{Sym}\left(m\right)$
is called primitive if it does not preserve any non-trivial block-structure.
In particular, if $m\ge3$, $G$ must be transitive.}, and then rely on (results depending on) the CFSG to deduce that
$Q_{p}$ is the full alternating or symmetric group. If we rely on
Theorem \ref{thm:strong approximation for square-free composite}
to deduce Corollary \ref{cor:transitive for composite with factors we handle},
the latter also becomes partly dependent on the CFSG. This can be
avoided, however, and to this aim we also give a proof that $\Gamma$
acts transitively on $X^{*}\left(n\right)$ assuming only that $Q_{p_{1}},\ldots,Q_{p_{k}}$
are primitive permutation groups, without using the CFSG (see Theorem
\ref{thm:if primitive then transitive on product} below). To sum
up, the only results depending on the CFSG are Theorem \ref{thm:3 mod 4: alternating given specific conditions},
Corollary \ref{cor:3 mod 4 - alternating for density-1}, and the
part of Theorem \ref{thm:Action on T2-system} relating to primes
$p\equiv3\left(4\right)$. In contrast, Theorems \ref{thm:1 mod 4}
and \ref{thm:strong approximation for square-free composite} and
Corollary \ref{cor:transitive for composite with factors we handle}
do not depend on the CFSG. We illustrate this in Figure \ref{fig:line of arguments}.
\end{rem}
Indeed, the following result does not depend on the CFSG:
\begin{thm}
\label{thm:if primitive then transitive on product}Let $n=p_{1}\cdots p_{k}$
be a product of distinct primes. If $Q_{p_{1}},\ldots,Q_{p_{k}}$
are \emph{primitive }permutation groups, then $\Gamma$ acts transitively
on $X^{*}\left(n\right)$.
\end{thm}
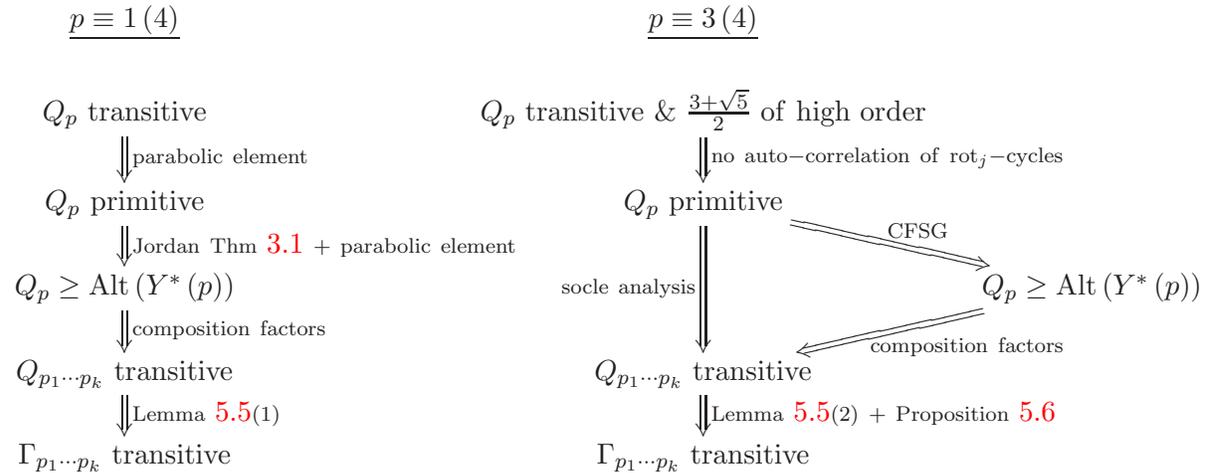
\begin{figure}[h]
\raggedright{}%
\begin{minipage}[t]{0.35\columnwidth}%
\[
\xymatrix@=15pt{\underline{p\equiv1\left(4\right)} &  & \underline{p\equiv3\left(4\right)}\\
Q_{p}~\mathrm{transitive\ar@{=>}[d]^{\mathrm{parabolic~element}}} &  & Q_{p}~\mathrm{transitive}~\&~\frac{3+\sqrt{5}}{2}~\mathrm{of~high~order}\ar@{=>}[d]^{\mathrm{no~auto-correlation~of}~\rot_{j}-\mathrm{cycles}}\\
Q_{p}~\mathrm{primitive}\ar@{=>}[d]^{\mathrm{Jordan~Thm~}\ref{thm:Jordan}\mathrm{~+~parabolic~element}} & ~~~~~~~~~~~~~~ & Q_{p}~\mathrm{primitive}\ar@{=>}[dd]_{\mathrm{socle~analysis}}\ar@{=>}[dr]^{~~\mathrm{CFSG}}\\
Q_{p}\ge\mathrm{Alt}\left(Y^{*}\left(p\right)\right)\ar@{=>}[d]^{\mathrm{composition~factors}} &  &  & Q_{p}\ge\mathrm{Alt}\left(Y^{*}\left(p\right)\right)\ar@{=>}[dl]^{\mathrm{~~~~~~~~composition~factors}}\\
Q_{p_{1}\cdots p_{k}}~\mathrm{transitive}\ar@{=>}[d]^{\mathrm{Lemma}~\ref{lem:Lambda transitive on X*(p_k)}(1)} &  & Q_{p_{1}\cdots p_{k}}~\mathrm{transitive}\ar@{=>}[d]^{\mathrm{Lemma}~\ref{lem:Lambda transitive on X*(p_k)}(2)\mathrm{~+~Proposition}~\ref{prop:(e,e,*) with order divisible by 4}}\\
\Gamma_{p_{1}\cdots p_{k}}~\mathrm{transitive} &  & \Gamma_{p_{1}\cdots p_{k}}~\mathrm{transitive}
}
\]
\end{minipage}\caption{\label{fig:line of arguments}The flow of arguments in the paper.
All the notions are explained in the sequel of the paper. Notice that
the case $p\equiv1\left(4\right)$ is indeed much simpler than its
counterpart $p\equiv3\left(4\right)$. To see the entire proof of
the results for primes $p\equiv1\left(4\right)$, it is enough to
read Section \ref{sec:Preliminaries}, the short Section \ref{sec:Alternating-Group-for 1(4)},
the short Section \ref{subsec:Transitivity-on-Y*(n)} and Section
\ref{subsec:Transitivity-on-X*(n)} up to the the first half of Lemma
\ref{lem:Lambda transitive on X*(p_k)}.}
\end{figure}

\subsection*{$T_{2}$-systems}

Let $G$ be a finitely generated group and $\F_{r}$ the free group
on $r$ generators. A normal subgroup $N\trianglelefteq\F_{r}$ is
said to be \emph{$G$-defining} if $\nicefrac{\F_{r}}{N}\cong G$.
Denote by $\Sigma_{r}\left(G\right)$\marginpar{$\Sigma_{r}\left(G\right)$}
the set of $G$-defining normal subgroups in $\F_{r}$. Consider the
action of $\mathrm{Aut}\left(\F_{r}\right)$ (in fact, of $\mathrm{Out}\left(\F_{r}\right)$)
on $\Sigma_{r}\left(G\right)$. The orbits of this action are called
$T_{r}$-systems of $G$. 

The following theorem is due to Gilman (for $r\ge4$) and Evans (who
extended to $r=3$):
\begin{thm}
\cite{gilman1977finite,evans1993t} \label{thm:Gilman}Let $G$ be
a finite non-abelian simple group and $r\ge3$. Then $\mathrm{Aut}\left(\F_{r}\right)$
acts on at least one $T_{r}$-system of $G$ as the alternating or
symmetric group.
\end{thm}
In fact, Gilman and Evans provide more information about the special
$T_{r}$-system on which $\mathrm{Aut}\left(\F_{r}\right)$ acts as
the full alternating or symmetric group, and show it is especially
large. Gilman also showed that for $G=\mathrm{PSL\left(2,p\right)}$
with $p\ge5$ prime, there is only one $T_{r}$-system for $r\ge3$.
Namely, he proved that $\mathrm{Aut}\left(\F_{r}\right)$ acts transitively
on $\Sigma_{r}\left(G\right)$. Theorem \ref{thm:Gilman} says, of
course, that the permutation group in this case is the alternating
or symmetric group. For more details we refer the reader to the beautiful
surveys \cite{pak2001PRA,lubotzky2011dynamics}.

When $r=2$, the action of $\mathrm{Aut}\left(\F_{2}\right)$ on $\Sigma_{2}\left(G\right)$
is not transitive for any finite non-abelian simple group $G$. In
fact, the number of $T_{2}$-systems tends to infinity as $\left|G\right|\to\infty$
\cite{garion2009commutator}. The main reason for this phenomenon
is that if $\left\{ a,b\right\} $ are a set of generators of $\F_{2}$,
and $\varphi\colon\F_{2}\twoheadrightarrow G$ an epimorphism, then
the set of conjugacy classes of\footnote{Here, $\left[a,b\right]$ denotes the commutator $aba^{-1}b^{-1}$.}
$\varphi\left(\left[a,b\right]\right)$ and of $\varphi\left(\left[a,b\right]\right)^{-1}$
is a well-defined invariant of the $G$-defining subgroup $N=\ker\varphi$,
which is also invariant under $\mathrm{Aut}\left(\F_{2}\right)$.
We elaborate more in Section \ref{sec:T-systems}.

Our result sheds more light on the case of $T_{2}$-systems for $G=\mathrm{PSL}\left(2,p\right)$.
If $A,B\in\mathrm{SL}\left(2,p\right)$ and we denote $x=\mathrm{tr}\left(A\right)$,
$y=\mathrm{tr}\left(B\right)$ and $z=\mathrm{tr}\left(AB\right)$,
then 
\[
\mathrm{tr}\left(\left[A,B\right]\right)=x^{2}+y^{2}+z^{2}-xyz-2.
\]
In Section \ref{sec:T-systems} it is explained why the map $\left(A,B\right)\mapsto\left(\mathrm{tr}\left(A\right),\mathrm{tr}\left(B\right),\mathrm{tr}\left(AB\right)\right)$
yields a bijection between the elements in $\Sigma_{2}\left(\mathrm{PSL\left(2,p\right)}\right)$
with associated trace $-2$ and the elements of $Y^{*}\left(p\right)$.
In this language, the main result of \cite{BGS-I2016markoff} \textendash{}
Theorem \ref{thm:BGS} above \textendash{} says that outside the exceptional
set of primes, these elements form a single $T_{2}$-system. See \cite{mccullough2013nielsen}
for an extensive survey of the connection between the Markoff equation
(\ref{eq:markoff}) and $T_{2}$-systems of $\mathrm{PSL}\left(2,p\right)$.
Through this connection, Theorems (\ref{thm:1 mod 4}) and (\ref{thm:3 mod 4: alternating given specific conditions})
translate to a result in the spirit of Theorem \ref{thm:Gilman}:
\begin{thm}
\label{thm:Action on T2-system}Assume that the prime $p$ satisfies
the assumptions of Theorem \ref{thm:1 mod 4} or of Theorem \ref{thm:3 mod 4: alternating given specific conditions}.
Then $\mathrm{Aut}\left(\F_{2}\right)$ acts on the trace-$\left(-2\right)$
$T_{2}$-system of $\mathrm{PSL}\left(2,p\right)$ as the full alternating
or symmetric group.
\end{thm}
The paper is organized as follows. Section \ref{sec:Preliminaries}
gives some more notation and collects some results from \cite{BGS-I2016markoff}
we use here. In the short Section \ref{sec:Alternating-Group-for 1(4)}
and longer Section \ref{sec:Alternating-Group-for 3 (4)} we prove
Theorem \ref{thm:1 mod 4} for $p\equiv1\left(4\right)$ and Theorem
\ref{thm:3 mod 4: alternating given specific conditions} for $p\equiv3\left(4\right)$,
respectively. Section \ref{sec:Strong-Approximation-for-composite}
is dedicated to proving the transitivity of $\Gamma$ in certain composite
moduli: first assuming the groups $Q_{p}$ contain the alternating
group (in Section \ref{subsec:Transitivity-on-Y*(n)}), and then assuming
only that $Q_{p}$ is primitive (Section \ref{subsec:transitivity on Y*(n) without CFSG}).
In Section \ref{sec:T-systems} we give some background on $T$-systems
and prove Theorem \ref{thm:Action on T2-system}. Finally, Appendix
\ref{sec:Dan-Carmon}, by Dan Carmon, shows that the assumption in
Theorem \ref{sec:Alternating-Group-for 3 (4)} regarding the order
of $\frac{3+\sqrt{5}}{2}\in\mathbb{F}_{p^{2}}$ holds for most primes. 

\subsection*{Acknowledgments}

We are indebted to Peter Sarnak for his encouragement, and for stimulating
discussions, enlightening suggestions and clever advice. We would
also like to thank Zeev Rudnick and P\"{a}r Kurlberg for beneficial comments,
and to Dan Carmon for writing the useful Appendix \ref{sec:Dan-Carmon}.
We have benefited much from the mathematical open source community,
and in particular from SageMath. Author Meiri was supported by BSF
grant 2014099 and ISF grant 662/15. Author Puder was supported by
the Rothschild fellowship, by the NSF under agreement No.~DMS-1128155
and by the ISF grant 1071/16. Author Carmon was supported by the European
Research Council under the European Union's Seventh Framework Programme
(FP7/2007-2013) / ERC grant agreement $\mathrm{n^{o}}$ 320755. 

\section{Preliminaries\label{sec:Preliminaries}}

Before proving our main results, let us describe some further notation
and collect further results from \cite{BGS-I2016markoff} that we
use below.

\subsection*{Further notation}
\begin{itemize}
\item We already introduced above the notation $\left[x,y,z\right]$ for
the block of the solution $\left(x,y,z\right)$ in $X^{*}\left(p\right)$,
so $\left[x,y,z\right]\in Y^{*}\left(p\right)$. We also use this
notation for a composite (square-free) modulo $n$: here $\left[x,y,z\right]$
is the element (block) in $Y^{*}\left(n\right)$ containing the solution
$\left(x,y,z\right)$.
\item Some elements in $\Gamma$ are permutations of the three coordinates
of solutions. We denote these elements by $\tau_{\left(12\right)}$\marginpar{$\tau_{\left(12\right)},\tau_{\left(123\right)}$}
for the permutation exchanging the first and second coordinates, by
$\tau_{\left(123\right)}$ for the cyclic permutation and so on. By
abuse of notation, we use the same notation for the corresponding
elements in $\Gamma$, $\Gamma_{p}$, $Q_{p}$, $\Gamma_{n}$ and
$Q_{n}$. 
\item The analysis in \cite{BGS-I2016markoff}, as well as in the current
work, relies heavily on three ``rotation'' elements\label{rotations}
$\rot_{1},\rot_{2},\rot_{3}\in\Gamma$\marginpar{$\protect\rot_{i}$}.
They are defined by
\[
\rot_{j}\stackrel{\mathrm{def}}{=}R_{j+2}\circ\tau_{\left(j+1~~j+2\right)}
\]
(the indices are taken modulo $3$). For example, $\left(x,y,z\right)\stackrel{\rot_{1}}{\mapsto}\left(x,z,xz-y\right)$.
The rotation $\rot_{j}$ fixes the $j$-th coordinate and its action
on $X^{*}\left(p\right)$ and on $Y^{*}\left(p\right)$ is completely
analyzed in \cite{BGS-I2016markoff} \textendash{} see Lemmas \ref{lem:rotations a la BGS p=00003D1(4)}
and \ref{lem:rotations a la BGS p=00003D3(4)} below. Again, by abuse
of notation we write $\rot_{i}$ for the rotation element in the different
groups $\Gamma$, $\Gamma_{p}$, $Q_{p}$, $\Gamma_{n}$ and $Q_{n}$. 
\item Following \cite{BGS-I2016markoff}, we denote the ``conic sections''
by $C_{j}\left(a\right)$ \marginpar{$C_{j}\left(a\right)$}, $j=1,2,3$.
These are defined as
\[
C_{j}\left(a\right)=\left\{ \left(x_{1},x_{2},x_{3}\right)\in X^{*}\left(p\right)\,\middle|\,x_{j}=a\right\} .
\]
When we write $C_{j}\left(\pm a\right)$\marginpar{$C_{j}\left(\pm a\right)$},
we mean the conic section in $Y^{*}\left(p\right)$:
\[
C_{j}\left(\pm a\right)=\left\{ \left[x_{1},x_{2},x_{3}\right]\in Y^{*}\left(p\right)\,\middle|\,x_{j}=a\right\} .
\]
\item For every prime $p$ we let $i$\marginpar{$i$} denote a square root
of $-1$ (in $\mathbb{F}_{p}$ or in $\mathbb{F}_{p^{2}}$).
\item For $x\in\nicefrac{\mathbb{Z}}{p\mathbb{Z}}$ we use the standard
Legendre symbol $\left(\frac{x}{p}\right)$\marginpar{$\left(\frac{x}{p}\right)$}
to denote the image of $x$ under the character of order 2. Namely,
\[
\left(\frac{x}{p}\right)=\begin{cases}
1 & x~\mathrm{is~a~square~in~\mathbb{F}_{p}^{*}}\\
-1 & x~\mathrm{is~a~non-square~in}~\mathbb{F}_{p}^{*}\\
0 & x=0
\end{cases}.
\]
\item The notation $\left|x\right|$ is used to denote the order of the
group element $x\in G$ in the group $G$.
\end{itemize}

\subsection*{Rotation elements}

The action of $\rot_{1}$ on the conic section $C_{1}\left(x\right)\subseteq X^{*}\left(p\right)$
is a linear map on the last two coordinates given by the matrix
\begin{equation}
\left(\begin{array}{cc}
0 & 1\\
-1 & x
\end{array}\right).\label{eq:rot1 as matrix}
\end{equation}
The eigenvalues of this matrix are given by $\frac{x\pm\sqrt{x^{2}-4}}{2}$.
This leads to the following definitions and lemmas from \cite{BGS-I2016markoff}:
\begin{defn}
\label{def:hyper-ellip-parab}
\begin{itemize}
\item An element $x\in\mathbb{F}_{p}$ if called \emph{hyperbolic}\marginpar{\emph{hyperbolic}}\emph{
}if $\left(x^{2}-4\right)$ is a square in $\mathbb{F}_{p}^{*}$.
\item An element $x\in\mathbb{F}_{p}$ if called \emph{elliptic}\marginpar{\emph{elliptic}}\emph{
}if $\left(x^{2}-4\right)$ is a non-square in $\mathbb{F}_{p}^{*}$.
\item An element $x\in\mathbb{F}_{p}$ if called \emph{parabolic}\marginpar{\emph{parabolic}}\emph{
}if $\left(x^{2}-4\right)=0$ in $\mathbb{F}_{p}$, namely, if $x=\pm2$.
\end{itemize}
\end{defn}
Notice that this categorization of the elements is invariant under
sign change $x\mapsto-x$. The following lemmas are based on Lemmas
3-5 of \cite{BGS-I2016markoff} which describe the action of $\rot_{i}$
on $X^{*}\left(p\right)$. We adapt them below in order to describe
the action of $\rot_{i}$ on $Y^{*}\left(p\right)$ and add some further
details, all follow easily from Section 2.1 in \cite{BGS-I2016markoff}.
We state the lemmas for $C_{1}\left(\pm x\right)$, but the same statements
holds, evidently, for $C_{2}\left(\pm x\right)$ and for $C_{3}\left(\pm x\right)$.
\begin{lem}
\label{lem:rotations a la BGS p=00003D1(4)}\cite[Lemmas 3-5]{BGS-I2016markoff}
Let $p\equiv1\left(4\right)$ be prime. Then,
\begin{itemize}
\item $\left|Y^{*}\left(p\right)\right|=\frac{p\left(p+3\right)}{4}.$
\item $\left|C_{1}\left(\pm2\right)\right|=p$; The permutation induced
by $\rot_{1}$ on $C_{1}\left(\pm2\right)$ consists of a single $p$-cycle. 
\item There are $\frac{p-1}{4}$ hyperbolic elements up to sign. For $x$
hyperbolic, $\left|C_{1}\left(\pm x\right)\right|=\frac{p-1}{2}$.
Let $\omega^{\pm1}\in\mathbb{F}_{p}$ be the eigenvalues of the matrix
(\ref{eq:rot1 as matrix}), so $x=\omega+\omega^{-1}$. The permutation
induced by $\rot_{1}$ on $C_{1}\left(\pm x\right)$ consists of $\frac{p-1}{2d}$
cycles of length $d$ each, where $d=\frac{\max\left(\left|\omega\right|,\left|-\omega\right|\right)}{2}$
and $\left|\omega\right|$ is the order of $\omega$ in the multiplicative
group $\mathbb{F}_{p}^{*}$. The solutions in $C_{1}\left(x\right)$
have the form $\left(x,\alpha+\beta,\alpha\omega+\beta\omega^{-1}\right)$
for $\alpha,\beta\in\mathbb{F}_{p}^{*}$ with $\alpha\beta=\frac{x^{2}}{x^{2}-4}$,
and 
\begin{equation}
\left(x,\alpha+\beta,\alpha\omega+\beta\omega^{-1}\right)\stackrel{\rot_{1}}{\mapsto}\left(x,\alpha\omega+\beta\omega^{-1},\alpha\omega^{2}+\beta\omega^{-2}\right).\label{eq:rot1 on solutions - hyper}
\end{equation}
\item There are $\frac{p-1}{4}$ elliptic elements up to sign. For $x$
elliptic, $\left|C_{1}\left(\pm x\right)\right|=\frac{p+1}{2}$. Define
$\omega$ as for hyperbolic elements by $x=\omega+\omega^{-1}$, only
now $\omega\in\mathbb{F}_{p^{2}}\setminus\mathbb{F}_{p}$. The permutation
induced by $\rot_{1}$ on $C_{1}\left(\pm x\right)$ consists of $\frac{p+1}{2d}$
cycles of length $d$ each, where $d=\frac{\max\left(\left|\omega\right|,\left|-\omega\right|\right)}{2}$
and $\left|\omega\right|$ is the order of $\omega$ in the multiplicative
group $\mathbb{F}_{p^{2}}^{~*}$. Moreover, $\omega^{p+1}=1$, i.e.~$\left|\omega\right|~|~\left(p+1\right)$.
The solutions in $C_{1}\left(x\right)$ have the form $\left(x,A+A^{p},A\omega+A^{p}\omega^{-1}\right)$
with $A\in\mathbb{F}_{p^{2}}^{*}$ and $A^{p+1}=\frac{x^{2}}{x^{2}-4}$,
and 
\begin{equation}
\left(x,A+A^{p},A\omega+A^{p}\omega^{-1}\right)\stackrel{\rot_{1}}{\mapsto}\left(x,A\omega+A^{p}\omega^{-1},A\omega^{2}+A^{p}\omega^{-2}\right).\label{eq:rot1 on solutions-ellip}
\end{equation}
\end{itemize}
\end{lem}
We sum up the content of Lemma \ref{lem:rotations a la BGS p=00003D1(4)}
in Table \ref{tab:1 mod 4}.

\begin{table}
\begin{tabular}{|>{\centering}p{2cm}|>{\centering}p{1.5cm}|>{\centering}p{1.5cm}|>{\centering}p{2.2cm}|>{\centering}p{7cm}|}
\hline 
type of $x$ & \# $x$'s up to sign & $\left|C_{1}\left(\pm x\right)\right|$ &  & cycle-structure for $\rot_{1}\Big|_{C_{1}\left(\pm x\right)}$\tabularnewline
\hline 
\hline 
parabolic & 1 & $p$ & $x=\pm2$ & a single $p$-cycle\tabularnewline
\hline 
hyperbolic (including $0$) & $\frac{p-1}{4}$ & $\frac{p-1}{2}$ & $x=\omega+\omega^{-1}$, $\omega\in\mathbb{F}_{p}^{*}\setminus\left\{ \pm1\right\} $ & For every $1\ne d~|~\frac{p-1}{2}$, there are $\left\lceil \frac{\varphi\left(d\right)}{2}\right\rceil $
hyperbolic $\pm x$ such that $\rot_{1}\Big|_{C_{1}\left(\pm x\right)}$
has $\frac{p-1}{2d}$ cycles of length $d$ each. (If $d$ is odd,
$\left|w\right|\in\left\{ d,2d\right\} $, if $d$ is even, $\left|w\right|=2d$.)\tabularnewline
\hline 
elliptic & $\frac{p-1}{4}$ & $\frac{p+1}{2}$ & $x=\omega+\omega^{-1}$ $\omega\in\mathbb{F}_{p^{2}}^{~*}\setminus\left\{ \pm1\right\} $
$\omega^{p+1}=1$ & For every $1\ne d~|~\frac{p+1}{2}$, there are $\frac{\varphi\left(d\right)}{2}$
elliptic $\pm x$ such that $\rot_{1}\Big|_{C_{1}\left(\pm x\right)}$
has $\frac{p+1}{2d}$ cycles of length $d$ each. ($\left|\omega\right|\in\left\{ d,2d\right\} $) \tabularnewline
\hline 
\end{tabular}\caption{The structure of $\protect\rot_{1}\in Q_{p}$ when $p\equiv1\left(4\right)$,
as follows from Lemma \ref{lem:rotations a la BGS p=00003D1(4)}.
In the rightmost column, every set $\left\{ x,-x\right\} $ is counted
once.\label{tab:1 mod 4}}
\end{table}

When $p\equiv3\left(4\right)$, our results are somewhat weaker and
the proofs more involved. The main reason for that is the lack of
solutions with the parabolic elements $\pm2$:
\begin{lem}
\label{lem:rotations a la BGS p=00003D3(4)}\cite[Lemmas 3-5]{BGS-I2016markoff}
Let $p\equiv3\left(4\right)$ be prime. Then,
\begin{itemize}
\item $\left|Y^{*}\left(p\right)\right|=\frac{p\left(p-3\right)}{4}$
\item There are no solutions in $Y^{*}\left(p\right)$ involving the parabolic
elements $\pm2$, nor the elliptic element $0$.
\item There are $\frac{p-3}{4}$ hyperbolic elements up to sign. For $x$
hyperbolic, the size and structure of $C_{1}\left(\pm x\right)$ and
the action of $\rot_{1}$ on $C_{1}\left(\pm x\right)$ have the same
properties as for $x$ hyperbolic when $p\equiv1\left(4\right)$ (see
Lemma \ref{lem:rotations a la BGS p=00003D1(4)}).
\item There are $\frac{p-3}{4}$ non-zero elliptic elements up to sign.
For $x$ elliptic, the size and structure of $C_{1}\left(\pm x\right)$
and the action of $\rot_{1}$ on $C_{1}\left(\pm x\right)$ have the
same properties as for $x$ elliptic when $p\equiv1\left(4\right)$
(see Lemma \ref{lem:rotations a la BGS p=00003D1(4)}).
\end{itemize}
\end{lem}
We sum up the content of Lemma \ref{lem:rotations a la BGS p=00003D3(4)}
in Table \ref{tab:3 mod 4}.

\begin{table}
\begin{tabular}{|>{\centering}p{2.3cm}|>{\centering}m{1.5cm}|>{\centering}m{1.4cm}|>{\centering}p{2.7cm}|>{\centering}p{6cm}|}
\hline 
type of $x$ & \# $x$'s up to sign & $\left|C_{1}\left(\pm x\right)\right|$ & eigenvalues of $\rot_{1}$ & cycle-structure of $\rot_{1}\Big|_{C_{1}\left(\pm x\right)}$\tabularnewline
\hline 
\hline 
hyperbolic $\left(\frac{x^{2}-4}{p}\right)=1$ & $\frac{p-3}{4}$ & $\frac{p-1}{2}$ & $\omega\in\mathbb{F}_{p}^{*}\setminus\left\{ \pm1\right\} $ $x=\omega+\omega^{-1}$ & For every $1\ne d~|~\frac{p-1}{2}$, there are $\frac{\varphi\left(d\right)}{2}$
hyperbolic $\pm x$ such that $\rot_{1}\Big|_{C_{1}\left(\pm x\right)}$
has $\frac{p-1}{2d}$ cycles of length $d$ each. ($\left|w\right|\in\left\{ d,2d\right\} $)\tabularnewline
\hline 
elliptic (exc.~$0$): $x\ne0$ \& $\left(\frac{x^{2}-4}{p}\right)=-1$ & $\frac{p-3}{4}$ & $\frac{p+1}{2}$ & $\omega\in\mathbb{F}_{p^{2}}^{*}\setminus\left\{ \pm1,\pm i\right\} $
$x=\omega+\omega^{-1}$ $\omega^{p+1}=1$ & For every $3\le d~|~\frac{p+1}{2}$, there are $\frac{\varphi\left(d\right)}{2}$
elliptic $\pm x$ such that $\rot_{1}\Big|_{C_{1}\left(\pm x\right)}$
has $\frac{p+1}{2d}$ cycles of length $d$ each. (If $d$ is odd,
$\left|\omega\right|\in\left\{ d,2d\right\} $, if $d$ is even, $\left|\omega\right|=2d$.)\tabularnewline
\hline 
\end{tabular}\caption{The structure of $\protect\rot_{1}\in Q_{p}$ when $p\equiv3\left(4\right)$,
as follows from Lemma \ref{lem:rotations a la BGS p=00003D3(4)}\label{tab:3 mod 4}}
\end{table}

For $x\in\mathbb{F}_{p}$, denote by $d_{p}\left(\pm x\right)$\marginpar{$d_{p}\left(\pm x\right)$}
the order of $\rot_{1}\in Q_{p}$ in its action on $C_{1}\left(\pm x\right)$.
Namely, the solutions with first coordinate $\pm x$ in $Y^{*}\left(p\right)$
belong to cycles of length $d_{p}\left(\pm x\right)$.

\section{Alternating Group for $p\equiv1\left(4\right)$ \label{sec:Alternating-Group-for 1(4)}}

This section contains the proof of Theorem \ref{thm:1 mod 4}, which
states that if $p\equiv1\left(4\right)$ and $Q_{p}$ is transitive,
then $Q_{p}$ contains the entire alternating group $\mathrm{Alt}\left(Y^{*}\left(p\right)\right)$.
As mentioned above, the existence of parabolic elements when $p\equiv1\left(4\right)$
allows a rather short argument in this case.

We use the following classical theorem of Jordan:
\begin{thm}[{Jordan \cite[Thm 3.3E]{dixon1996permutation}}]
\label{thm:Jordan} Let $G\le\mathrm{Sym}\left(n\right)$ be a primitive
permutation group containing a cycle of prime length $p\le n-3$.
Then $G\ge\mathrm{Alt}\left(n\right)$.
\end{thm}
\begin{proof}[Proof of Theorem \ref{thm:1 mod 4}]
 Assume $p\equiv1\left(4\right)$, and let $\rot_{1}\in Q_{p}$ be
the rotation element defined on Page \pageref{rotations}. This element
has one $p$-cycle, while all its other cycles have length coprime
to $p$ (see Table \ref{tab:1 mod 4}). Thus its power $\sigma=\rot_{1}^{~\left|\rot_{1}\right|/p}\in Q_{p}$
is a $p$-cycle. As $\left|Y^{*}\left(p\right)\right|=\frac{p\left(p+3\right)}{4}\ge p+3$,
it is now sufficient to show, by Jordan's Theorem (Theorem \ref{thm:Jordan}
above), that $Q_{p}$ is primitive in $\mathrm{Sym}\left(Y^{*}\left(p\right)\right)$. 

We need to show that the group $Q_{p}$ preserves no non-trivial block
structure. Assume there is a block structure $\left\{ B_{1},\ldots,B_{m}\right\} $
preserved by $Q_{p}$. So $\bigcup B_{i}=Y^{*}\left(p\right)$ and
$B_{i}\cap B_{j}=\emptyset$ for $i\ne j$, and for every $g\in Q_{p}$
and every $i$, $g\left(B_{i}\right)=B_{j}$ for some $j$.

Consider $C_{1}\left(\pm2\right)\subset Y^{*}\left(p\right)$, the
$p$ elements contained in the cycle of size $p$ in $\sigma$. The
set $C_{1}\left(\pm2\right)$ must be contained in a block, for otherwise
it has to be the union of several equally-sized blocks, but $p$ is
prime. Say $C_{1}\left(\pm2\right)\subseteq B_{1}$. So $B_{1}$ contains
all solutions with $\pm2$ in the first coordinate. In particular,
it contains $\left[2,2,2+2i\right]$ and $\left[2,2+2i,2\right]$.
But the same argument with $\rot_{2}$ and $\rot_{3}$ shows that
$B_{1}$ contains all solutions with $\pm2$ in any coordinate. So
$B_{1}$ is invariant under all three rotations and under all permutations
of coordinates, and therefore invariant under the action of the whole
group $Q_{p}$. By the transitivity of $Q_{p}$, $B_{1}=Y^{*}\left(p\right)$.
\end{proof}
\begin{rem}
The proof of Theorem \ref{thm:BGS} in \cite{BGS-I2016markoff} shows
that for \emph{every prime $p$}, the large component of $X^{*}\left(p\right)$
contains all solutions with parabolic $\left(\pm2\right)$ coordinates.
Thus, our proof of Theorem \ref{thm:1 mod 4} applies to the general
case: the group $\Gamma$ acts on the large component of $Y^{*}\left(p\right)$
as the alternating or symmetric group.
\end{rem}

\section{Alternating Group for $p\equiv3$ $\left(4\right)$\label{sec:Alternating-Group-for 3 (4)}}

In the case where $p\equiv3\left(4\right)$, there are no parabolic
elements, and in Sections \ref{subsec:Primitivity-for-p=00003D3(4)}
and \ref{subsec:properties of blocks} we establish the primitivity
of $Q_{p}$ for density-1 of these primes rather than for all those
outside the exceptional set from Theorem \ref{thm:BGS}. We also rely
on much deeper theorems, involving the classification of finite simple
groups (CFSG), to conclude in Section \ref{subsec:primitivity =00003D=00003D> Alt}
that whenever $Q_{p}$ is primitive, it contains $\mathrm{Alt}\left(Y^{*}\left(p\right)\right)$.
Throughout this section, we assume that $p\equiv3\left(4\right)$.

\subsection{Primitivity of $Q_{p}$ when $p\equiv3$$\left(4\right)$\label{subsec:Primitivity-for-p=00003D3(4)}}

In this subsection we prove that under the assumptions of Theorem
\ref{thm:3 mod 4: alternating given specific conditions}, the permutation
group $Q_{p}$ is primitive. Namely,
\begin{thm}
\label{thm:primitivity for p=00003D3(4)}Let $p$ be prime with $p\equiv3\left(4\right)$.
Assume that $Q_{p}$ is transitive and that the order of $\frac{3+\sqrt{5}}{2}\in\mathbb{F}_{p^{2}}$
is at least $32\sqrt{p+1}$. Then $Q_{p}$ is primitive.
\end{thm}
To establish primitivity of $Q_{p}$, one needs to show there are
no non-trivial blocks in the action of $Q_{p}$ on $Y^{*}\left(p\right)$:
a block is a subset $B\subseteq Y^{*}\left(p\right)$, such that for
every $g\in Q_{p}$, either $g.B=B$ or $g.B\cap B=\emptyset$. As
$Q_{p}$ is assumed to be transitive, if $B$ is proper ($B\subsetneqq Y^{*}\left(p\right)$)
and of size at least two, then the subsets $\left\{ g.B\,\middle|\,g\in Q_{p}\right\} $
constitute a partition of $Y^{*}\left(p\right)$ which is a non-trivial
block structure preserved under the action of $Q_{p}$. So proving
$Q_{p}$ is primitive is equivalent to showing that every proper block
is a singleton.

The proof of Theorem \ref{thm:primitivity for p=00003D3(4)} relies
on the following two propositions which contain properties of blocks
in $Y^{*}\left(p\right)$. We defer the proofs of these two propositions
to the next subsection, and complete the proof of Theorem \ref{thm:primitivity for p=00003D3(4)}
in the current subsection, assuming the two propositions.

We say that some coordinate $j\in\left\{ 1,2,3\right\} $ is \emph{homogeneous}
in a block $B\subseteq Y^{*}\left(p\right)$ if the $j$-th coordinate
of every solution in $B$ is of the same type (either all hyperbolic
or all elliptic). 
\begin{prop}
\label{prop:2 homogenous coordinates}Let $p\equiv3\left(4\right)$.
Assume that $Q_{p}$ acts transitively on $Y^{*}\left(p\right)$,
and let $B\subsetneqq Y^{*}\left(p\right)$ be a proper $Q_{p}$-block.
Then at least two of the coordinates $\left\{ 1,2,3\right\} $ are
homogeneous in $B$.
\end{prop}
The most technical ingredient of the proof of primitivity is the following.
Recall that $d_{p}\left(\pm x\right)$ denotes the length of the cycles
of $\rot_{1}\in Q_{p}$ containing elements of $C_{1}\left(\pm x\right)$.
\begin{prop}
\label{prop:high-order =00003D=00003D> no self-correlation}Assume
that $Q_{p}$ is transitive and let $x\in\mathbb{F}_{p}\setminus\left\{ 0,\pm2\right\} $
satisfy $d_{p}\left(\pm x\right)\ge16\sqrt{p+1}$. Then, for every
$j\in\left\{ 1,2,3\right\} $, every proper $Q_{p}$-block $B\subsetneqq Y^{*}\left(p\right)$
contains at most one solution with $j$-th coordinate $\pm x$.
\end{prop}
The idea of the proof of this proposition is the following: assume
there are two solutions in the block $B$ with first coordinate $\pm x$.
Say these are $\left[x,y_{0},y_{1}\right]$ and $\left[x,z_{0},z_{1}\right]$.
Then for every $1\le m$, the block $\rot_{1}^{~m}\left(B\right)$
contains the solutions $\left[x,y_{m},y_{m+1}\right]$ and $\left[x,z_{m},z_{m+1}\right]$
with $y_{m}$ and $z_{m}$ defined recursively by $y_{m+1}=xy_{m}-y_{m-1}$
and $z_{m+1}=xz_{m}-z_{m-1}$. By Proposition \ref{prop:2 homogenous coordinates},
at least one of the two coordinates $2,3$ in every block is homogeneous,
meaning that for every $m$, either $y_{m}$ and $z_{m}$ have the
same type (hyperbolic or elliptic), or $y_{m+1}$ and $z_{m+1}$ have
the same type. Using classical results in number theory, we show such
``high correlation'' between two cycles of $\rot_{1}$ is impossible
whenever these cycles are long enough.

Section \ref{subsec:properties of blocks} gives the details of the
proof, and assuming it, we finish the proof of Theorem \ref{thm:primitivity for p=00003D3(4)}.
We need the following corollary showing that elements of high order
in the sense of Proposition \ref{prop:high-order =00003D=00003D> no self-correlation}
appear in the same block and the same coordinate only with other elements
of the same type and the same order:
\begin{cor}
\label{cor:coordinates are homogenous in order of elements}Assume
that $Q_{p}$ is transitive and that $x\in\mathbb{F}_{p}\setminus\left\{ 0,\pm2\right\} $
satisfies $d_{p}\left(\pm x\right)\ge16\sqrt{p+1}$. If $B\subsetneqq Y^{*}\left(p\right)$
is a proper $Q_{p}$-block containing some solution with first coordinate
$\pm x$, and another solution with first coordinate $\pm x'$, then
$d_{p}\left(\pm x\right)=d_{p}\left(\pm x'\right)$. In particular,
$x$ and $x'$ are of the same type (both hyperbolic or both elliptic).
\end{cor}
\begin{proof}
Note that $\rot_{1}^{~d_{p}\left(\pm x\right)}\left(B\right)=B$.
By Proposition \ref{prop:high-order =00003D=00003D> no self-correlation},
$\rot_{1}^{~m}\left(B\right)\ne B$ for $1\le m<d_{p}\left(\pm x\right).$
Hence, $d_{p}\left(\pm x'\right)$ is some multiple of $d_{p}\left(\pm x\right)$.
In particular, the assumption of Proposition \ref{prop:high-order =00003D=00003D> no self-correlation}
holds for $x'$, and by symmetry, $d_{p}\left(\pm x\right)$ is a
multiple of $d_{p}\left(\pm x'\right)$. Hence $d_{p}\left(\pm x'\right)=d_{p}\left(\pm x\right)$.
\end{proof}

\begin{proof}[Proof of Theorem \ref{thm:primitivity for p=00003D3(4)} assuming
Propositions \ref{prop:2 homogenous coordinates} and \ref{prop:high-order =00003D=00003D> no self-correlation}]
 Assume that $Q_{p}$ is transitive and $\omega=\frac{3+\sqrt{5}}{2}\in\mathbb{F}_{p^{2}}^{~*}$
has order at least $32\sqrt{p+1}$. We need to show that $Q_{p}$
is primitive. We use the special symmetric solution $\left[3,3,3\right]\in Y^{*}\left(p\right)$.
Whenever $\omega\in\mathbb{F}_{p^{2}}$ has high order in the multiplicative
group $\mathbb{F}_{p^{2}}^{*}$, the cycle of $\rot_{1}\in Q_{p}$
containing the solution $\left[3,3,3\right]$ is long. More concretely,
$3=\omega+\omega^{-1}$, and by Lemma \ref{lem:rotations a la BGS p=00003D3(4)}
and Table \ref{tab:3 mod 4}, $d_{p}\left(\pm3\right)$ is either
$\left|\omega\right|$ or $\frac{\left|\omega\right|}{2}$, where
$\left|\omega\right|$ is the order of $\omega$ in the multiplicative
group $\mathbb{F}_{p^{2}}^{~~*}$. So $d_{p}\left(\pm3\right)\ge16\sqrt{p+1}$.

Assume that $\left[a,b,c\right]$ and $\left[3,3,3\right]$ are two
distinct solutions lying in the same proper $Q_{p}$-block $B\subsetneqq Y^{*}\left(p\right)$.
By Lemma \ref{lem:rotations a la BGS p=00003D3(4)}, $d_{p}\left(\pm3\right)\ge16\sqrt{p+1}$,
and by Corollary \ref{cor:coordinates are homogenous in order of elements},
$d_{p}\left(\pm a\right)=d_{p}\left(\pm b\right)=d_{p}\left(\pm c\right)=d_{p}\left(\pm3\right).$
As $\left[3,3,3\right]$ is the only solution of the form $\left[x,x,x\right]$
or $\left[x,x,-x\right]$, we can assume without loss of generality
that $\left\{ \pm b\right\} \ne\left\{ \pm c\right\} $. Since $\tau_{\left(2~3\right)}$
stabilizes $\left[3,3,3\right]$, we have $\tau_{\left(2~3\right)}\left(B\right)=B$,
so the two distinct solutions $\left[a,b,c\right]$ and $\left[a,c,b\right]$
both belong to $B$. This contradicts Proposition \ref{prop:high-order =00003D=00003D> no self-correlation}:
$d_{p}\left(\pm a\right)=d_{p}\left(\pm3\right)$ is large and thus
$a$ cannot appear twice in the same coordinate in the same block.
\end{proof}
As mentioned in Section \ref{sec:Introduction}, the assumptions in
Theorem \ref{thm:primitivity for p=00003D3(4)} hold for density-1
of the primes $p\equiv3\left(4\right)$. Indeed, relying on strong
results of Ford \cite{ford2008distribution}, Dan Carmon proves in
Proposition \ref{prop:main appendix prop} in Appendix \ref{sec:Dan-Carmon}
that under some assumptions, the order of a quadratic integer modulo
primes is high for density-1 of the primes. From Proposition \ref{prop:main appendix prop}
we deduce:
\begin{cor}
\label{cor:3 has high order in density 1}For density-1 of all primes,
the element $\omega=\frac{3+\sqrt{5}}{2}\in\mathbb{F}_{p^{2}}$ has
order at least $32\sqrt{p+1}$ in the multiplicative group $\mathbb{F}_{p^{2}}^{~~*}$,
in which case $d_{p}\left(\pm3\right)\ge16\sqrt{p+1}$.
\end{cor}
Combining Theorem \ref{thm:BGS} with Corollary \ref{cor:3 has high order in density 1}
shows why the assumptions in Theorem \ref{thm:primitivity for p=00003D3(4)}
hold for density-1 of all primes $p\equiv3\left(4\right)$, hence:
\begin{cor}
\label{cor:primitivity for density 1 for p=00003D3(4)}For density-1
of all primes $p\equiv3\left(4\right)$, the group $Q_{p}$ is primitive
in its action on $Y^{*}\left(p\right)$.
\end{cor}
\begin{rem}
\label{rem:no-correlation conjecture}It is conceivable that there
is a stronger version of Proposition \ref{prop:high-order =00003D=00003D> no self-correlation}
which states there cannot be correlation between two long cycles of
$\rot_{1}\in Q_{p}$ even with two different first coordinates. Were
we able to prove this, we could omit the condition about the order
of $\frac{3+\sqrt{5}}{2}$ in the statements of Theorems \ref{thm:3 mod 4: alternating given specific conditions}
and \ref{thm:primitivity for p=00003D3(4)} and assume only that $Q_{p}$
is transitive to conclude that it is primitive and, moreover, contains
$\mathrm{Alt}\left(Y^{*}\left(p\right)\right)$. (This would make
Theorem \ref{thm:3 mod 4: alternating given specific conditions}
completely parallel to Theorem \ref{thm:1 mod 4} dealing with $p\equiv1\left(4\right)$.)
In Remark \ref{rem:no correlation between cycles of different x's}
below we explain the obstacle to proving this more general version
of Proposition \ref{prop:high-order =00003D=00003D> no self-correlation}.
\end{rem}

\subsection{Properties of blocks in the action of $Q_{p}$ on $Y^{*}\left(p\right)$\label{subsec:properties of blocks}}

In the current subsection we prove the two propositions that were
stated without proof in the previous subsection. Proposition \ref{prop:2 homogenous coordinates}
is proved in Section \ref{subsec:Homogeneity-of-coordinates}, and
Proposition \ref{prop:high-order =00003D=00003D> no self-correlation}
proved in Sections \ref{subsec:No-correlation-hyper} (the hyperbolic
case) and \ref{subsec:No-correlation-ellip} (the elliptic case).

\subsubsection{Homogeneity of coordinates in blocks\label{subsec:Homogeneity-of-coordinates}}
\begin{lem}
\label{lem:<rot1,rot2,rot3> of index <=00003D 2}The subgroup $H=\left\langle \rot_{1},\rot_{2},\rot_{3}\right\rangle \le\Gamma$
has index at most $2$ in $\Gamma$.
\end{lem}
\begin{proof}
By definition, $\Gamma$ is generated by the three Vieta involutions
and permutations of coordinates. Since $R_{3}=\rot_{1}\cdot\tau_{\left(2~3\right)}$
and likewise for $R_{1}$ and $R_{2}$, since $\tau_{\left(1~3~2\right)}=\rot_{3}\cdot\rot_{1}$
and since $S_{3}=\left\langle \left(12\right),\left(132\right)\right\rangle $,
we obtain that $\Gamma=\left\langle \rot_{1},\rot_{2},\rot_{3},\tau_{\left(1~2\right)}\right\rangle =\left\langle H,\tau_{\left(1~2\right)}\right\rangle $.
It is easy to check that $\tau_{\left(1~2\right)}\rot_{j}\tau_{\left(1~2\right)}\in H$
for $j=1,2,3$, so $H\trianglelefteq\Gamma$ and $\Gamma=H\cdot\left\langle \tau_{\left(1~2\right)}\right\rangle $.
This finishes the proof.
\end{proof}
Recall that Proposition \ref{prop:2 homogenous coordinates} says
that if $Q_{p}$ acts transitively on $Y^{*}\left(p\right)$, and
if $B\subsetneqq Y^{*}\left(p\right)$ is a proper block of the action
of $Q_{p}$ on $Y^{*}\left(p\right)$, then at least two of the coordinates
$\left\{ 1,2,3\right\} $ are homogeneous in $B$.
\begin{proof}[Proof of Proposition \ref{prop:2 homogenous coordinates}]
Assume that some coordinate, say $j=1$, is not homogeneous in $B$.
We need to show that the second and third coordinates are homogeneous.
The element $\rot_{1}^{\left(p-1\right)/2}$ fixes every solution
with first coordinate hyperbolic, while $\rot_{1}^{\left(p+1\right)/2}$
fixes every solution with first coordinate elliptic. Hence $B$ is
invariant under both elements, and thus by $\rot_{1}$. 

By the same argument, if all three coordinates are not homogeneous,
$B$ is invariant under $H_{p}=\left\langle \rot_{1},\rot_{2},\rot_{3}\right\rangle \le Q_{p}$.
By Lemma \ref{lem:<rot1,rot2,rot3> of index <=00003D 2}, $\left[Q_{p}:H_{p}\right]\le2$,
and transitivity implies there are at most two blocks in the action:
$B$ and $B'=\gamma\left(B\right)$ for some $\gamma\in Q_{p}$. But
the block containing $\left[3,3,3\right]$ is also invariant under
$\tau_{\left(1~2\right)}$, hence is invariant under the whole of
$Q_{p}$ \textendash{} a contradiction.

Thus at least one coordinate \textendash{} the second or the third
\textendash{} is homogeneous. Notice that $\rot_{1}$, which stabilizes
$B$, moves the third coordinate of the solutions to the second. Hence
both the second and third coordinates must be homogeneous. 
\end{proof}
\begin{rem}
In fact, the proof of the last lemma yields something slightly stronger.
Denote the type of a solution in $Y^{*}\left(p\right)$ by some triple
in $\left\{ h,e\right\} ^{3}$, depending on whether every coordinate
is hyperbolic or elliptic. Then, every block $B$ as above contains
either only solutions of the same type (homogeneous in all coordinates),
or only solutions of exactly two types: one type is $\left(h,h,h\right)$
or $\left(e,e,e\right)$, and the other differs from the first type
in one coordinate (the sole non-homogeneous coordinate).
\end{rem}

\subsubsection{No correlation between two long $\protect\rot_{1}$-cycles with the
same first hyperbolic coordinate\label{subsec:No-correlation-hyper}}

We now prove Proposition \ref{prop:high-order =00003D=00003D> no self-correlation}
stating that if $Q_{p}$ is transitive and $d_{p}\left(\pm x\right)\ge16\sqrt{p+1}$,
then $\pm x$ cannot appear twice in the same coordinate in the same
proper $Q_{p}$-block $B\subsetneqq Y^{*}\left(p\right)$. What we
actually prove is the lack of correlation between two long enough
cycles of $\rot_{j}$ with the same $j$-th coordinate (including
the case of two different offsets of the same cycle). The proof of
Proposition \ref{prop:high-order =00003D=00003D> no self-correlation}
is split to the case where $x$ is hyperbolic (in the current subsection)
and the case it is elliptic (given in Section \ref{subsec:No-correlation-ellip}).

We use the following classical number-theoretic result: 
\begin{thm}[{Weil \cite[Theorem II.2C']{schmidt1976equations}}]
\label{thm:Weil} Let $f\left(x\right)\in\mathbb{F}_{p}\left[x\right]$
be a polynomial with $m$ distinct roots in $\overline{\mathbb{F}_{p}}$
which is not a square in $\overline{\mathbb{F}_{p}}\left[x\right]$.
Then 
\[
\left|\sum_{s\in\mathbb{F}_{p}}\left(\frac{f\left(s\right)}{p}\right)\right|\le\left(m-1\right)\sqrt{p}.
\]
\end{thm}
\begin{proof}[Proof of Proposition \ref{prop:high-order =00003D=00003D> no self-correlation}
when $x$ is hyperbolic]
Assume that $x$ is hyperbolic with $d_{p}\left(\pm x\right)\ge16\sqrt{p+1}$,
and that there are two elements in the proper $Q_{p}$-block $B\subsetneqq Y^{*}\left(p\right)$
with $\pm x$ in the first coordinate. The same arguments holds, evidently,
for every coordinate $j=1,2,3$.

Assume that $\left[x,y_{0},y_{1}\right]$ and $\left[x,z_{0},z_{1}\right]$
belong to $B$. By Lemma \ref{lem:rotations a la BGS p=00003D3(4)},
$x=\omega+\omega^{-1}$ with $\omega\in\mathbb{F}_{p}^{*}$ and we
can assume $\left|\omega\right|=2d\ge32\sqrt{p-1}$: otherwise, replace
$x$ with $-x$ and $\omega$ with $-\omega$. Write $y_{0}=\alpha+\beta$
with $\alpha,\beta\in\mathbb{F}_{p}^{*}$ so that $\alpha\beta=\frac{x^{2}}{x^{2}-4}$
and $y_{1}=\alpha\omega+\beta\omega^{-1}$ (see Lemma \ref{lem:rotations a la BGS p=00003D3(4)}).
The cycle of $\rot_{1}$ containing $\left[x,y_{0},y_{1}\right]$
is 
\[
\left[x,y_{0},y_{1}\right]=\left[x,y_{d},y_{d+1}\right],\left[x,y_{1},y_{2}\right],\ldots,\left[x,y_{d-2},y_{d-1}\right],\left[x,y_{d-1},y_{d}\right]
\]
with
\[
y_{j}=\alpha\omega^{j}+\beta\omega^{-j}.
\]
The set $\left\{ \omega^{j}\right\} _{0\le j\le2d-1}$ is the same
as the set $\left\{ s^{m}\right\} _{s\in\mathbb{F}_{p}^{*}}$ where
$m=\frac{p-1}{2d}$ (with every element in $\left\{ \omega^{j}\right\} $
covered by $\frac{p-1}{2d}$ different values of $s$). So as sets,
\[
\left\{ y_{0},\ldots,y_{2d-1}\right\} =\left\{ \alpha\omega^{j}+\beta\omega^{-j}\right\} _{0\le j\le2d-1}=\left\{ f_{\alpha,\beta}\left(s\right)\stackrel{\mathrm{def}}{=}\alpha s^{m}+\beta s^{-m}\right\} _{s\in\mathbb{F}_{p}^{*}}.
\]
The same holds for the cycle of $\rot_{1}$ containing $\left[x,z_{0},z_{1}\right]$
with $\gamma,\delta\in\mathbb{F}_{p}^{*}$ in the role of $\alpha,\beta$,
so that $z_{j}=\gamma\omega^{j}+\delta\omega^{-j}$. We may assume
that $\gamma\ne\pm\alpha$, for otherwise $\left[x,y_{0},y_{1}\right]=\left[x,z_{0},z_{1}\right]$.
Moreover, if $s^{m}=\omega^{j}$ then $f_{\alpha,\beta}\left(s\right)=y_{j}$
and $f_{\gamma,\delta}\left(s\right)=z_{j}$. 

Notice that $y_{j}$ and $z_{j}$ are of different types (one hyperbolic
and the other elliptic) if and only if 
\begin{equation}
\left(\frac{\left(y_{j}^{2}-4\right)\left(z_{j}^{2}-4\right)}{p}\right)=-1.\label{eq:different types}
\end{equation}
Since $\left[x,y_{j},y_{j+1}\right]$ and $\left[x,z_{j},z_{j+1}\right]$
both belong to the block $\rot_{1}^{~j}\left(B\right)$, we derive
from Proposition \ref{prop:2 homogenous coordinates} that (\ref{eq:different types})
cannot hold for two consecutive values of $j$. In the parametrization
given by $s\in\mathbb{F}_{p}^{*}$, this means that 
\begin{equation}
\left(\frac{\left(f_{\alpha,\beta}\left(s\right)^{2}-4\right)\left(f_{\gamma,\delta}\left(s\right)^{2}-4\right)}{p}\right)=\left(\frac{\left(f_{\alpha\omega,\beta\omega^{-1}}\left(s\right)^{2}-4\right)\left(f_{\gamma\omega,\delta\omega^{-1}}\left(s\right)^{2}-4\right)}{p}\right)=-1\label{eq:contradiction goal}
\end{equation}
cannot hold for any $s\in\mathbb{F}_{p}^{*}$.

Write 
\[
g_{\alpha,\beta}\left(s\right)\stackrel{\mathrm{def}}{=}\left(s^{m}\right)^{2}\left(f_{\alpha,\beta}\left(s\right)^{2}-4\right)=\left[\left(\alpha s^{2m}+\beta\right)^{2}-4s^{2m}\right]\in\mathbb{F}_{p}\left[s\right],
\]
and $k_{1}\left(s\right)=g_{\alpha,\beta}\left(s\right)g_{\gamma,\delta}\left(s\right)$
and $k_{2}\left(s\right)=g_{\alpha\omega,\beta\omega^{-1}}\left(s\right)g_{\gamma\omega,\delta\omega^{-1}}\left(s\right)$.
Now (\ref{eq:contradiction goal}) is equivalent to 
\begin{equation}
\left(\frac{k_{1}\left(s\right)}{p}\right)=\left(\frac{k_{2}\left(s\right)}{p}\right)=-1.\label{eq:contradiction goal with k1 k2}
\end{equation}
Denote by $N_{\left(-1,-1\right)}$ the number of $s\in\mathbb{F}_{p}$
for which (\ref{eq:contradiction goal with k1 k2}) holds. Our goal
is to show that $N_{\left(-1,-1\right)}>0$, whence (\ref{eq:contradiction goal with k1 k2})
has some solution $s\ne0$, yielding a contradiction (note that $s=0$
is not a solution to (\ref{eq:contradiction goal with k1 k2})). Note
that $k_{1}\left(s\right),k_{2}\left(s\right)\ne0$ for every $s\in\mathbb{F}_{p}$:
indeed, $g_{\alpha,\beta}\left(0\right)=\beta^{2}\ne0$, and if $0\ne s\in\mathbb{F}_{p}$
and $g_{\alpha,\beta}\left(s\right)=0$ then $f_{\alpha,\beta}\left(s\right)=\pm2$
is $y_{j}$ for some $j$, but there are no solution in $X^{*}\left(p\right)$
containing $\pm2$ when $p\equiv3\left(4\right)$. Therefore $\left(\frac{k_{1}\left(s\right)}{p}\right),\left(\frac{k_{2}\left(s\right)}{p}\right)\ne0$
for $s\in\mathbb{F}_{p}$ and 
\begin{eqnarray}
N_{\left(-1,-1\right)} & = & \frac{1}{4}\sum_{s\in\mathbb{F}_{p}}\left(1-\left(\frac{k_{1}\left(s\right)}{p}\right)\right)\left(1-\left(\frac{k_{2}\left(s\right)}{p}\right)\right)\nonumber \\
 & = & \frac{1}{4}\left[p-\sum_{p\in\mathbb{F}_{p}}\left(\frac{k_{1}\left(s\right)}{p}\right)-\sum_{p\in\mathbb{F}_{p}}\left(\frac{k_{2}\left(s\right)}{p}\right)+\sum_{p\in\mathbb{F}_{p}}\left(\frac{k_{1}\left(s\right)k_{2}\left(s\right)}{p}\right)\right].\label{eq:long formula for N(-1,-1) hyper}
\end{eqnarray}
For every $\emptyset\ne B\subseteq\left\{ 1,2\right\} $, define 
\begin{equation}
M_{B}\stackrel{\mathrm{def}}{=}\sum_{s\in\mathbb{F}_{p}}\left(\frac{\prod_{j\in B}k_{j}\left(s\right)}{p}\right).\label{eq:def of MB - hyper}
\end{equation}
Then (\ref{eq:long formula for N(-1,-1) hyper}) becomes 
\begin{equation}
N_{\left(-1,-1\right)}=\frac{1}{4}\left[p-M_{\left\{ 1\right\} }-M_{\left\{ 2\right\} }+M_{\left\{ 1,2\right\} }\right].\label{eq:formula for (-1,-1)-hyper}
\end{equation}
We use Theorem \ref{thm:Weil} to estimate the $M_{B}$'s. First,
we show that none of $k_{1},k_{2}$ and $k_{1}k_{2}$ are squares
in $\overline{\mathbb{F}_{p}}\left[x\right]$. The roots of 
\[
g_{\alpha,\beta}\left(s\right)=\left(\alpha s^{2m}+\beta+2s^{m}\right)\left(\alpha s^{2m}+\beta-2s^{m}\right)
\]
satisfy 
\[
s^{m}=\frac{\pm2\pm\sqrt{4-4\alpha\beta}}{2\alpha}=\frac{\pm1\pm\sqrt{1-\frac{x^{2}}{x^{2}-4}}}{\alpha}=\frac{\pm1\pm\sqrt{\frac{-4}{x^{2}-4}}}{\alpha}.
\]
As $x$ is hyperbolic and $p\equiv3$$\left(4\right)$, we have that
$\frac{-4}{x^{2}-4}$ is not a square in $\mathbb{F}_{p}$, so $1$
and $\sqrt{\frac{-4}{x^{2}-4}}$ are linearly independent over $\mathbb{F}_{p}$,
and $\frac{\pm1\pm\sqrt{\frac{-4}{x^{2}-4}}}{\alpha}$ are four distinct
values for $S^{m}$, different from zero. Moreover, the polynomial
$s^{m}-\xi$ is separable for $0\ne\xi\in\mathbb{F}_{p^{2}}$ because
$m=\frac{p-1}{2d}<p$. So $g_{\alpha,\beta}\left(s\right)$, which
is of degree $4m$, has $4m$ distinct roots in $\overline{\mathbb{F}_{p}}$,
and in particular is not a square in $\overline{\mathbb{F}_{p}}\left[x\right]$. 

This analysis shows that $g_{\alpha,\beta}$ and $g_{\gamma,\delta}$
have a common root if and only if $\alpha=\pm\gamma$. Since $\alpha\ne\pm\gamma$
by assumption, $k_{1}=g_{\alpha,\beta}g_{\gamma,\delta}$ and $k_{2}=g_{\alpha\omega,\beta\omega^{-1}}g_{\gamma\omega,\delta\omega^{-1}}$
are both separable of degree $8m$. Finally, $k_{1}k_{2}$, of degree
$16m$, is also not a square in $\overline{\mathbb{F}_{p}}\left[x\right]$:
for $\alpha\ne\pm\alpha\omega$ and if $\alpha=\pm\gamma\omega$ then
$\alpha\omega\ne\pm\gamma$.

Theorem \ref{thm:Weil} yields that $\left|M_{\left\{ 1\right\} }\right|,\left|M_{\left\{ 2\right\} }\right|\le\left(8m-1\right)\sqrt{p}$
and $\left|M_{\left\{ 1,2\right\} }\right|\le\left(16m-1\right)\sqrt{p}$.
From (\ref{eq:formula for (-1,-1)-hyper}) we now obtain
\begin{eqnarray*}
N_{\left(-1,-1\right)} & \ge & \frac{1}{4}\left[p-2\left(8m-1\right)\sqrt{p}-\left(16m-1\right)\sqrt{p}\right]\\
 & = & \frac{1}{4}\left[p-32m\sqrt{p}+3\sqrt{p}\right]\\
 & \stackrel{m=\frac{p-1}{2d}}{=} & \frac{1}{4}\left[p-\frac{16\left(p-1\right)}{d}\sqrt{p}+3\sqrt{p}\right]\\
 & \stackrel{d\ge16\sqrt{p-1}}{\ge} & \frac{1}{4}\left[p-\sqrt{\left(p-1\right)p}+3\sqrt{p}\right]>\frac{3\sqrt{p}}{4}>0.
\end{eqnarray*}
\end{proof}

\subsubsection{No correlation between two long $\protect\rot_{1}$-cycles with the
same first elliptic coordinate\label{subsec:No-correlation-ellip}}

The general proof strategy for the elliptic case is the same as for
the hyperbolic case, albeit with a few extra technical details. In
the hyperbolic case, we used a parametrization of the elements of
a cycle of $\rot_{1}$ as a function over $\mathbb{F}_{p}^{*}$, which
allowed us to use Weil's bound (Theorem \ref{thm:Weil} above). In
the elliptic case, a similar approach requires that we go over the
elements in the cyclic subgroup of size $p+1$ in $\mathbb{F}_{p^{2}}^{\,*}$.
The following lemma allows us to parametrize this subgroup as a function
over $\mathbb{F}_{p}$:
\begin{lem}
\label{lem:H}The multiplicative subgroup\marginpar{H} $H\le\mathbb{F}_{p^{2}}^{\,\,*}$
of order $p+1$ satisfies
\begin{equation}
H=\left\{ \theta+i\eta\,\middle|\,\theta,\eta\in\mathbb{F}_{p},\,\,\theta^{2}+\eta^{2}=1\right\} =\left\{ \frac{2s}{1+s^{2}}+i\frac{1-s^{2}}{1+s^{2}}\,\middle|\,s\in\mathbb{F}_{p}\right\} \cup\left\{ -i\right\} \label{eq:H}
\end{equation}
(where $i=\sqrt{-1}\in\mathbb{F}_{p^{2}}$).
\end{lem}
\begin{proof}
Note that $\left(\theta+i\eta\right)^{p}=\theta-i\eta$ (recall that
$p\equiv3\left(4\right)$ so $i^{p}=i^{4k+3}=i^{3}=-i$). So $\left(\theta+i\eta\right)^{p+1}=\left(\theta+i\eta\right)\left(\theta-i\eta\right)=\theta^{2}+\eta^{2}.$
This gives the first equality in (\ref{eq:H}). A straightforward
computation yields the second equality.
\end{proof}

\begin{proof}[Proof of Proposition \ref{prop:high-order =00003D=00003D> no self-correlation}
when $x$ is elliptic]
We assume that $x$ is elliptic with\linebreak{}
$d_{p}\left(\pm x\right)\ge16\sqrt{p+1}$, and assume that there are
two elements in the proper $Q_{p}$-block $B\subsetneqq Y^{*}\left(p\right)$
with $\pm x$ in the first coordinate. We use the notation $H$ for
the subgroup of order $p+1$ in $\mathbb{F}_{p^{2}}^{~*}$, as in
Lemma \ref{lem:H}. Assume that $\left[x,y_{0},y_{1}\right]$ and
$\left[x,z_{0},z_{1}\right]$ both belong to $B$. By Table \ref{tab:3 mod 4},
$x=\omega+\omega^{-1}$ with $\omega\in H$, and we can assume that
$\left|\omega\right|=2d\ge32\sqrt{p+1}$, for otherwise replace $\omega$
by $-\omega$ and $x$ by $-x$. Let $A\in\mathbb{F}_{p^{2}}$ satisfy
that $A^{p+1}=\frac{x^{2}}{x^{2}-4}$, that $y_{0}=A+A^{p}$ and that
$y_{1}=A\omega+A^{p}\omega^{-1}$ (see Lemma \ref{lem:rotations a la BGS p=00003D3(4)}).
The cycle of $\rot_{1}$ containing $\left[x,y_{0},y_{1}\right]$
is 
\[
\left[x,y_{0},y_{1}\right]=\left[x,y_{d},y_{d+1}\right],\left[x,y_{1},y_{2}\right],\ldots,\left[x,y_{d-2},y_{d-1}\right],\left[x,y_{d-1},y_{d}\right]
\]
with
\[
y_{j}=A\omega^{j}+A^{p}\omega^{-j}.
\]
The set $\left\{ \omega^{j}\right\} _{0\le j\le2d-1}$ is the same
as the set $\left\{ h^{m}\right\} _{h\in H}$ where $m=\frac{p+1}{2d}$,
with every element in $\left\{ \omega^{j}\right\} $ covered by $m$
different values of $h$. So as sets, 
\[
\left\{ y_{0},\ldots,y_{2d-1}\right\} =\left\{ A\omega^{j}+A^{p}\omega^{-j}\right\} _{0\le j\le2d-1}=\left\{ f_{A}\left(h\right)\stackrel{\mathrm{def}}{=}Ah^{m}+A^{p}h^{-m}\right\} _{h\in H}.
\]
The same holds for the cycle of $\rot_{1}$ containing $\left[x,z_{0},z_{1}\right]$
with $C\in\mathbb{F}_{p^{2}}$ in the role of $A$, so that $z_{j}=C\omega^{j}+C^{p}\omega^{-j}$.
We may assume that $C\ne\pm A$, for otherwise $\left[x,y_{0},y_{1}\right]=\left[x,z_{0},z_{1}\right]$.
Moreover, if $h^{m}=\omega^{j}$ then $f_{A}\left(h\right)=y_{j}$
and $f_{C}\left(h\right)=z_{j}$. 

As in the proof of the hyperbolic case, we derive from Proposition
\ref{prop:2 homogenous coordinates} that 
\begin{equation}
\left(\frac{\left(f_{A}\left(h\right)^{2}-4\right)\left(f_{C}\left(h\right)^{2}-4\right)}{p}\right)=\left(\frac{\left(f_{A\omega}\left(h\right)^{2}-4\right)\left(f_{C\omega}\left(h\right)^{2}-4\right)}{p}\right)=-1\label{eq:contradiction goal-2}
\end{equation}
cannot hold for any $h\in H$. To be able to use Theorem \ref{thm:Weil},
we want to reparametrize (\ref{eq:contradiction goal-2}) as polynomials
in $s\in\mathbb{F}_{p}$, using Lemma \ref{lem:H}. Denote
\[
g_{A}\left(s\right)\overset{\mathrm{def}}{=}\left(1+s^{2}\right)^{2m}\left[f_{A}\left(h\left(s\right)\right)^{2}-4\right]
\]
where 
\[
h\left(s\right)=\frac{2s+i\left(1-s^{2}\right)}{1+s^{2}}=\frac{-i\left(s+i\right)^{2}}{1+s^{2}}=\frac{-i\left(s+i\right)}{\left(s-i\right)}.
\]
Let also $k_{1}=g_{A}g_{C}$ and $k_{2}=g_{A\omega}g_{C\omega}$.
Then (\ref{eq:contradiction goal-2}) is equivalent to 
\begin{equation}
\left(\frac{k_{1}\left(s\right)}{p}\right)=\left(\frac{k_{2}\left(s\right)}{p}\right)=-1.\label{eq:contradiction goal-3}
\end{equation}
As in the proof of the hyperbolic case, denote by $N_{\left(-1,-1\right)}$
the number of $s\in\mathbb{F}_{p}$ for which (\ref{eq:contradiction goal-3})
holds. Our goal is to get a contradiction by showing that $N_{\left(-1,-1\right)}>0$.
Note that $g_{A}\left(s\right)\ne0$ for $s\in\mathbb{F}_{p}$ because
$g_{A}\left(s\right)=\left(1+s^{2}\right)^{2m}\left(y_{j}^{2}-4\right)$
for some $y_{j}$ as above, and $s\ne\pm i$ and $y_{j}\ne\pm2$.
Thus $k_{i}\left(s\right)\ne0$ neither, and $\left(\frac{k_{i}\left(s\right)}{p}\right)\in\left\{ 1,-1\right\} $.
As in equations (\ref{eq:long formula for N(-1,-1) hyper})-\ref{eq:formula for (-1,-1)-hyper}
in the hyperbolic case, we get that 
\begin{equation}
N_{\left(-1,-1\right)}=\frac{1}{4}\left[p-M_{\left\{ 1\right\} }-M_{\left\{ 2\right\} }+M_{\left\{ 1,2\right\} }\right],\label{eq:formula for (-1,-1)-ellip}
\end{equation}
where for $\emptyset\ne B\in\left\{ 1,2\right\} $, we define $M_{B}\stackrel{\mathrm{def}}{=}\sum_{s\in\mathbb{F}_{p}}\left(\frac{\prod_{j\in B}k_{j}\left(s\right)}{p}\right)$.

We use Theorem \ref{thm:Weil} to estimate the $M_{B}$'s. First,
we show that $k_{1},k_{2}\in\mathbb{F}_{p}\left[x\right]$. Notice
that 
\[
h\left(s\right)^{-1}=\frac{\left(s-i\right)}{-i\left(s+i\right)}=\frac{i\left(s-i\right)}{\left(s+i\right)},
\]
so 
\begin{eqnarray}
g_{A}\left(s\right) & = & \left(1+s^{2}\right)^{2m}\left[f_{A}\left(h\left(s\right)\right)+2\right]\left[f_{A}\left(h\left(s\right)\right)-2\right]\nonumber \\
 & = & \left(1+s^{2}\right)^{2m}\left(Ah\left(s\right)^{m}+A^{p}h\left(s\right)^{-m}+2\right)\left(Ah\left(s\right)^{m}+A^{p}h\left(s\right)^{-m}-2\right)\label{eq:g in terms of h}\\
 & = & \left(A\left[-i\left(s+i\right)^{2}\right]^{m}+A^{p}\left[i\left(s-i\right)^{2}\right]^{m}+2\left[1+s^{2}\right]^{m}\right)\cdot\nonumber \\
 &  & \cdot\left(A\left[-i\left(s+i\right)^{2}\right]^{m}+A^{p}\left[i\left(s-i\right)^{2}\right]^{m}-2\left[1+s^{2}\right]^{m}\right).\label{eq:g as poly in s}
\end{eqnarray}
The last expression shows that $g_{A}\left(s\right)\in\mathbb{F}_{p^{2}}\left[s\right]$.
Its degree is $4m$: indeed, the leading coefficient is 
\[
\left(-1\right)^{m}\left(A^{2}+A^{2p}\right)+2A^{p+1}-4,
\]
and for $m$ even this coefficient equals $\left(A+A^{p}\right)^{2}-4=y_{0}^{2}-4$
which is not zero since $y_{0}\ne\pm2$ (see Lemma \ref{lem:rotations a la BGS p=00003D3(4)}).
For $m$ odd, this coefficient is 
\[
-\left(A+A^{p}\right)^{2}+4\left(A^{p+1}-1\right)=-y_{0}^{2}+4\left(A^{p+1}-1\right),
\]
which is not zero because $A^{p+1}-1=\frac{4}{x^{2}-4}$ is not a
square in $\mathbb{F}_{p}$ when $x$ is elliptic. 

As $\mathbb{F}_{p^{2}}=\mathbb{F}_{p}+i\mathbb{F}_{p}$, we can write
$g_{A}=g'_{A}+ig''_{A}$, where $g'_{A},g''_{A}\in\mathbb{F}_{p}\left[s\right]$.
By definition, for every $s\in\mathbb{F}_{p}$, we have $h=h\left(s\right)\in H$,
and 
\[
g_{A}\left(s\right)=\left(1+s^{2}\right)^{2m}\left[f_{A}\left(h\right)^{2}-4\right]\in\mathbb{F}_{p}
\]
so $g''_{A}\left(s\right)=0$ for every $s\in\mathbb{F}_{p}$. Since
$\deg\left(g''_{A}\right)\le4m<p$, we conclude that $g''_{A}$ is
the zero polynomial, hence $g_{A}\left(s\right)=g_{A}'\left(s\right)\in\mathbb{F}_{p}\left[s\right]$
and so $k_{1},k_{2}\in\mathbb{F}_{p}\left[x\right]$.

Next, we wish to show that $k_{1}$, $k_{2}$ and $k_{1}k_{2}$ are
not squares in $\overline{\mathbb{F}_{p}}\left[x\right]$. There is
a one-to-one correspondence between the roots of $g_{A}$ in $\overline{\mathbb{F}_{p}}$
and the roots of
\[
r_{A}\left(h\right)\stackrel{\mathrm{def}}{=}\left(Ah^{2m}+2h^{m}+A^{p}\right)\left(Ah^{2m}-2h^{m}+A^{p}\right),
\]
in $\overline{\mathbb{F}_{p}}$ given by 
\begin{eqnarray*}
\alpha & \mapsto & h\left(\alpha\right)=\frac{-i\left(\alpha+i\right)}{\left(\alpha-i\right)}\\
\frac{1+ih}{h+i}=\alpha\left(h\right) & \mapsfrom & h,
\end{eqnarray*}
because $\pm i$ is never a root of $g_{A}$ (recall that $g_{A}\left(s\right)\in\mathbb{F}_{p}\left[s\right]$
has the form from (\ref{eq:g as poly in s})) and $-i$ never a root
of $r_{A}$ (because $r_{A}\left(h\right)=h^{2m}\left(f_{A}\left(h\right)^{2}-4\right)$,
$-i\in H$ and thus $f_{A}\left(-i\right)=y_{j}$ for some $y_{j}$
as above, and $y_{j}\ne\pm2$). It is easier to analyze the roots
of $r_{A}$ than those of $g_{A}$: if $h$ is a root of $r_{A}$
then
\[
h^{m}=\frac{\pm1\pm\sqrt{1-\kappa\left(x\right)}}{A},
\]
where $\kappa\left(x\right)=A^{p+1}=\frac{x^{2}}{x^{2}-4}$. Now note
the following:

\begin{itemize}
\item The four possible values of $h^{m}$ are distinct and different from
zero (this follows from $\kappa\left(x\right)\ne0$,1).
\item Because $\left(m,p\right)=1$, the four polynomials $h^{m}-\frac{\pm1\pm\sqrt{1-\kappa\left(x\right)}}{A}$
are separable, so $r_{A}$ has $4m$ distinct roots in $\overline{\mathbb{F}_{p}}$,
and so does $g_{A}$.
\item If $A\ne\pm C$, the $4m$ roots of $r_{A}$ are distinct from the
$4m$ roots of $r_{C}$: certainly $\frac{1+\sqrt{1-\kappa\left(x\right)}}{A}\ne\pm\frac{1+\sqrt{1-\kappa\left(x\right)}}{C}$,
and if $\frac{1+\sqrt{1-\kappa\left(x\right)}}{A}=\pm\frac{1-\sqrt{1-\kappa\left(x\right)}}{C}$
we obtain
\begin{eqnarray*}
C & = & \pm A\cdot\frac{1-\sqrt{1-\kappa\left(x\right)}}{1+\sqrt{1-\kappa\left(x\right)}}\\
\kappa\left(x\right)=C^{p+1} & = & A^{p+1}\left(\frac{1-\sqrt{1-\kappa\left(x\right)}}{1+\sqrt{1-\kappa\left(x\right)}}\right)^{p+1}=\kappa\left(x\right)\xi^{p+1}
\end{eqnarray*}
with $\xi=\frac{1-\sqrt{1-\kappa\left(x\right)}}{1+\sqrt{1-\kappa\left(x\right)}}\in\mathbb{F}_{p}$
because $1-\kappa\left(x\right)=\frac{-4}{x^{2}-4}$ is a square in
$\mathbb{F}_{p}$. Then $\xi=\pm1$, that is, $C=\pm A$ \textendash{}
a contradiction. Hence $k_{1}=g_{A}g_{C}$ and $k_{2}=g_{A\omega}g_{C\omega}$
are separable of degree $8m$ each.
\item Finally, if $C\ne\pm A$, the polynomial $k_{1}k_{2}=g_{A}g_{A\omega}g_{C}g_{C\omega}$
is not a square in $\overline{\mathbb{F}_{p}}\left[x\right]$: it
is separable unless $A=\pm C\omega$ or $A\omega=\pm C$, but the
two cannot hold simultaneously.
\end{itemize}
We can now apply Theorem \ref{thm:Weil} to obtain the same bounds
on the $M_{B}$'s as in the hyperbolic case, and from (\ref{eq:formula for (-1,-1)-ellip})
we now obtain
\begin{eqnarray*}
N_{\left(-1,-1\right)} & \ge & \frac{1}{4}\left[p-2\left(8m-1\right)\sqrt{p}-\left(16m-1\right)\sqrt{p}\right]\\
 & = & \frac{1}{4}\left[p-32m\sqrt{p}+3\sqrt{p}\right]\\
 & \stackrel{m=\frac{p+1}{2d}}{=} & \frac{1}{4}\left[p-\frac{16\left(p+1\right)}{d}\sqrt{p}+3\sqrt{p}\right]\\
 & \stackrel{d\ge16\sqrt{p+1}}{\ge} & \frac{\sqrt{p}}{4}\left[\sqrt{p}-\sqrt{\left(p+1\right)}+3\right]>0.
\end{eqnarray*}

\end{proof}
\begin{rem}
\label{rem:no correlation between cycles of different x's}As we noted
in Remark \ref{rem:no-correlation conjecture} above, it is conceivable
that a stronger version of Proposition \ref{prop:high-order =00003D=00003D> no self-correlation}
holds. Let us point to the phase in the current argument that fails
in this more general setting. The simplest case to consider if that
of $x,x'\in\mathbb{F}_{p}$ both hyperbolic of maximal order, so $d_{p}\left(x\right)=d_{p}\left(x'\right)=\frac{p-1}{2}$.
Assume that $x=\omega+\omega^{-1}$ and $x'=\omega'+\omega'^{-1}$,
and that $\omega'=\omega^{r}$. Then, in the notation of Section \ref{subsec:No-correlation-hyper},
if $y_{j}=\alpha s+\beta s^{-1}$, then $y_{j}'=\alpha's^{r}+\beta's^{-r}$,
and our goal is to show that $\left(\alpha s+\beta s^{-1}\right)$
and $\left(\alpha's^{r}+\beta's^{-r}\right)$ cannot be of the same
type (hyperbolic/elliptic) for too many values of $s\in\mathbb{F}_{p}^{*}$.
The problem is that $r$ can be of any order, and is generically of
order $\ge\sqrt{p}$. For polynomials of such degree Weil's Theorem
\ref{thm:Weil} is useless. 
\end{rem}

\subsection{Deducing Alternating group from primitivity\label{subsec:primitivity =00003D=00003D> Alt}}

Finally, in this section, we show how to deduce that $Q_{p}\ge\alt\left(Y^{*}\left(p\right)\right)$
whenever $Q_{p}$ is primitive. Throughout this section we denote
the symmetric group $\mathrm{Sym}\left(n\right)$ by $S_{n}$ and
$\mathrm{Alt}\left(n\right)$ by $A_{n}$\marginpar{$S_{n},A_{n}$}.
Here we use the following result of Guralnick and Magaard, classifying
primitive subgroups of $S_{n}$ containing an element with at least
$n/2$ fixed points. This theorem relies heavily on the CFSG. We adjust
the statement of the theorem to our needs \textendash{} the original
statement in \cite{guralnick1998minimal} is more detailed. In the
statement we use the notation $\mathrm{Soc}\left(G\right)$ for the
socle of the group $G$ (see Section \ref{subsec:transitivity on Y*(n) without CFSG}
for details), and the standard notation $G_{1}\wr G_{2}$ for the
wreath product of two groups.
\begin{thm}[{\cite[Theorem 1]{guralnick1998minimal}}]
\label{thm:Guralnik-Magaard} Let $G\le\mathrm{S_{n}}$ be a primitive
group, and let $x\in G$ have at least $n/2$ fixed points. Then one
of the following holds:

\begin{enumerate}
\item \label{enu:affine case}$G=\mathrm{Aff}\left(2,k\right)$ is the affine
group acting on $\mathbb{F}_{2}^{~k}$ and $x$ is a transvection\footnote{To be sure, $x$ is a transvection when $\mathrm{Aff}\left(2,k\right)$
is embedded in $\mathrm{GL}\left(2,k+1\right)$ as the matrices with
bottom row $\left(0,\ldots,0,1\right)$.} and is, in particular, an involution. In this case $x$ has exactly
$n/2$ fixed points.
\item \label{enu:action on subsets case}There are $r\ge1$, $m\ge5$ and
$1\le k\le m/4$ such that $n=\binom{m}{k}^{r}$, the group $S_{m}$
acts on the set $\Delta$ of $k$-subsets of $\left\{ 1,\ldots,m\right\} $
in the natural way, $G\le S_{m}\wr S_{r}$ acts on $\Delta^{r}$ and
$\mathrm{Soc}\left(G\right)=A_{m}^{~r}$.
\item \label{enu:specail S_6 case}For some $r\ge1$, $n=6^{r}$, the group
$S_{6}$ acts on $\Delta=\left\{ 1,\ldots,6\right\} $ by applying
an outer automorphism\footnote{Namely, for some fixed $\varphi\in\mathrm{Aut}\left(S_{6}\right)\setminus\mathrm{Inn}\left(S_{6}\right)$,
the permutation $\sigma\in S_{6}$ acts on $\Delta$ by $\sigma.i=\varphi\left(\sigma\right)\left(i\right)$.}, $G\le S_{6}\wr S_{r}$ acts on $\Delta^{r}$ and $\mathrm{Soc}\left(G\right)=A_{6}^{~r}$.
\item \label{enu:orthogonal case}The group $G$ is some variant of an orthogonal
group over the field of two elements acting on some collection of
$1$-spaces or hyperplanes, and the element $x$ is an involution.
\end{enumerate}
\end{thm}
The following lemma helps us rule out Case \ref{enu:action on subsets case}
of the above theorem with $r=1$.
\begin{lem}
\label{lem:property of action on subsets}Consider the embedding $\iota\colon S_{m}\hookrightarrow S_{n}$
given by the natural action of the symmetric group $S_{m}$ on the
set $\Delta$ of $n=\binom{m}{k}$ $k$-subsets of $m$, for some
$2\le k\le\frac{m}{4}$. If, for some $\pi\in S_{m}$, the image $\iota\left(\pi\right)$
has a cycle of size divisible by $q$ and a cycle of size divisible
by $s$ for some distinct primes $q$ and $s$, then $\iota\left(\pi\right)$
also has a cycle of size divisible by $qs$.
\end{lem}
\begin{proof}
Assume that $\left\{ a_{1},\ldots,a_{k}\right\} \in\Delta$ belongs
to a cycle $\alpha$ of length divisible by $q$ in $\iota\left(\pi\right)$.
Assume that in $\pi$, the elements $a_{1},\ldots,a_{k}$ belong to
$t$ distinct cycles: the elements $a_{1},\ldots,a_{\ell_{1}}$ belong
to the cycle $\sigma_{1}$, the elements $a_{\ell_{1}+1},\ldots,a_{\ell_{2}}$
belong to the cycle $\sigma_{2}$, and so on (each $\sigma_{j}$ may
contain additional elements not from $\left\{ a_{1},\ldots,a_{k}\right\} $).
Let $o_{1}$ be the smallest power of $\sigma_{1}$ that maps $\left\{ a_{1},\ldots,a_{\ell_{1}}\right\} $
to itself. Define $o_{2},\ldots,o_{t}$ analogously. Then, $q~|\thinspace\mathrm{lcm}\left(o_{1},\ldots,o_{t}\right)$.
In particular, $q~|~o_{i}$ for some $i$, and so $q\Big|\left|\sigma_{i}\right|$.
Without loss of generality, assume $q\Big|\left|\sigma_{1}\right|$,
so that $a_{1}$ belongs to a cycle $\sigma=\sigma_{1}$ of $\pi$
of size divisible by $q$. Likewise, assume that $b_{1}$ belongs
to a cycle $\tau$ of $\pi$ of size divisible by $s$.

Denote $A=\left\{ 1,\ldots,m\right\} \setminus\left(\sigma\cup\tau\right)$
(namely, $A$ consists of the elements not belonging to the cycle
$\sigma$ nor to $\tau$). Assume first that $\sigma\ne\tau$. If
$\left|A\right|\ge k-2$, then a $k$-subset containing $a_{1}$,
$b_{1}$ and $k-2$ elements from $A$ belongs to a cycle of $\iota\left(\pi\right)$
of size divisible by $qs$. If $\left|A\right|<k-2$, then, as $k\le\frac{m}{4}$,
at least one of $\sigma$ or $\tau$ has more than $k$ element. Assume
without loss of generality it is $\sigma$. Consider the $k$-subset
$\left\{ b_{1},a_{1},\pi\left(a_{1}\right),\pi^{2}\left(a_{1}\right),\ldots,\pi^{k-2}\left(a_{1}\right)\right\} $.
This subset belongs to a cycle of $\iota\left(\pi\right)$ of size
$\mathrm{lcm}\left(\left|\tau\right|,\left|\sigma\right|\right)$,
which, in particular, is a multiple of $qs$.

Finally, assume $\sigma=\tau$. Then $qs\Big|\left|\sigma\right|$.
If the length of $\sigma$ is at least $k+1$, the $k$-subset $\left\{ a_{1},\pi\left(a_{1}\right),\pi^{2}\left(a_{1}\right),\ldots,\pi^{k-2}\left(a_{1}\right),\pi^{k-1}\left(a_{1}\right)\right\} $
belongs to a cycle of $\iota\left(\pi\right)$ of size dividing $qs$.
If $\left|\sigma\right|\le k$ then $A$ contains more than $k-1$
elements, and the $k$-subset containing $a_{1}$ and $k-1$ elements
from $A$ belongs to a cycle of $\iota\left(\pi\right)$ of size dividing
$qs$.
\end{proof}
\begin{prop}
\label{prop:primitive =00003D=00003D> Alternating} Let $p\equiv3\left(4\right)$
be prime. If $Q_{p}$ is primitive, then $Q_{p}\ge\alt\left(Y^{*}\left(p\right)\right)$.
\end{prop}
\begin{proof}
Consider $\rot_{1}\in Q_{p}$. Among the $\frac{p\left(p-3\right)}{4}$
elements in $Y^{*}\left(p\right)$, $\frac{\left(p-1\right)\left(p-3\right)}{8}$
belong to cycles of length at least $3$ and dividing $\frac{p-1}{2}$,
and $\frac{\left(p+1\right)\left(p-3\right)}{8}$ belong to cycles
of length at least $3$ and dividing $\frac{p+1}{2}$ (see Table \ref{tab:3 mod 4}).
Since $\mathrm{gcd}\left(\frac{p-1}{2},\frac{p+1}{2}\right)=1$, the
permutation $\sigma=\rot_{1}^{\left(p+1\right)/2}$ fixes exactly
$\frac{\left(p+1\right)\left(p-3\right)}{8}>\frac{\left|Y^{*}\left(p\right)\right|}{2}$
elements of $Y^{*}\left(p\right)$. Thus $Q_{p}$ satisfies the assumptions
in Theorem \ref{thm:Guralnik-Magaard}. We can now rule out all options
except for $Q_{p}=\alt\left(Y^{*}\left(p\right)\right)$ or $Q_{p}=\mathrm{Sym}\left(Y^{*}\left(p\right)\right)$.

Cases \ref{enu:affine case} and \ref{enu:orthogonal case} are immediately
ruled out because the permutation $\sigma\in Q_{p}$ is not an involution.
Case \ref{enu:action on subsets case} with $r\ge2$ and Case \ref{enu:specail S_6 case}
are immediately ruled out because $\left|Y^{*}\left(p\right)\right|=\frac{p\left(p-3\right)}{4}$
is not a proper power nor equal to six. It remains to consider Case
\ref{enu:action on subsets case} with $r=1$. 

Let $q$ be some prime factor of $\frac{p-1}{2}$, and let $s$ be
some prime factor of $\frac{p+1}{2}$. By Table \ref{tab:3 mod 4},
$\rot_{1}$ contains cycles of size divisible by $q$ (indeed, even
of size $q$ exactly), and of size divisible by $s$. However, it
does not contain any cycle of size divisible by $qs$. Using Lemma
\ref{lem:property of action on subsets}, this rules out Case \ref{enu:action on subsets case}
from Theorem \ref{thm:Guralnik-Magaard} with $r=1$ and $k\ge2$.
The remaining case, that of Case \ref{enu:action on subsets case}
with $r=k=1$, is precisely the case that the group in question is
either $A_{n}$ or $S_{n}$. 
\end{proof}
This finishes the proofs of Theorem \ref{thm:3 mod 4: alternating given specific conditions}
and of Corollary \ref{cor:3 mod 4 - alternating for density-1}: Theorem
\ref{thm:3 mod 4: alternating given specific conditions} is now a
consequence of Theorem \ref{thm:primitivity for p=00003D3(4)} and
Proposition \ref{prop:primitive =00003D=00003D> Alternating}, while
Corollary \ref{cor:3 mod 4 - alternating for density-1} follows from
Corollary \ref{cor:primitivity for density 1 for p=00003D3(4)} and
Proposition \ref{prop:primitive =00003D=00003D> Alternating}.

\section{Strong Approximation for Square Free Composite Moduli\label{sec:Strong-Approximation-for-composite}}

In this section we derive our main application of the results on the
groups $Q_{p}$ and show that $\Gamma$ acts transitively on $X^{*}\left(n\right)$
for various square-free composite values $n=p_{1}\cdots p_{k}$. First,
in Section \ref{subsec:Transitivity-on-Y*(n)}, we prove that if $Q_{p_{j}}\ge\mathrm{Alt}\left(Y^{*}\left(p_{j}\right)\right)$
for every $j=1,\ldots,k$, then $\Gamma$ acts transitively on $Y^{*}\left(n\right)$.
In Section \ref{subsec:Transitivity-on-X*(n)} we strengthen this
result to showing that, moreover, $\Gamma$ acts transitively on $X^{*}\left(n\right)$,
namely, that strong approximation for the Markoff equation holds in
modulo $n$, thus proving Theorem \ref{thm:strong approximation for square-free composite}.

At this point, we are able to prove Theorem \ref{thm:3 mod 4: alternating given specific conditions}
that $Q_{p}\ge\mathrm{Alt}\left(Y^{*}\left(p\right)\right)$ for $p\equiv3\left(4\right)$
satisfying the assumptions in the statement of Theorem \ref{thm:3 mod 4: alternating given specific conditions},
only while relying on the classification of finite simple groups (CFSG)
\textendash{} see Section \ref{subsec:primitivity =00003D=00003D> Alt}.
However, the CFSG is not necessary for establishing the transitivity
of $\Gamma$ on $X^{*}\left(n\right)$ when $n=p_{1}\cdots p_{k}$
and $p_{1},\ldots,p_{k}$ are distinct primes satisfying the assumptions
in Theorems \ref{thm:1 mod 4} or \ref{thm:3 mod 4: alternating given specific conditions}
(this is Corollary \ref{cor:transitive for composite with factors we handle}).
In Section \ref{subsec:transitivity on Y*(n) without CFSG} we give
an alternative proof for the transitivity of $\Gamma$ on $X^{*}\left(n\right)$,
which uses only the primitivity of $Q_{p}$, as in Theorem \ref{thm:primitivity for p=00003D3(4)},
thus proving Theorem \ref{thm:if primitive then transitive on product}.
The point is that we want to provide a proof of the transitivity on
$X^{*}\left(n\right)$ which can be potentially understood in full,
from basic principles, by a motivated reader. This is practically
impossible if one relies on the CFSG.

\subsection{Transitivity of $\Gamma$ on $Y^{*}\left(n\right)$\label{subsec:Transitivity-on-Y*(n)}}

Here we prove the following lemma:
\begin{lem}
\label{lem:alt =00003D=00003D> transitivity on Y*(n)}Let $n=p_{1}\cdots p_{k}$
be a product of distinct primes. If $Q_{p_{j}}\ge\mathrm{Alt}\left(Y^{*}\left(p_{j}\right)\right)$
for $j=1,\ldots,k$, then $\Gamma$ acts transitively on $Y^{*}\left(n\right)$.\\
Moreover, $Q_{n}$, which is a subgroup of $\mathrm{Sym}\left(Y^{*}\left(p_{1}\right)\right)\times\ldots\times\mathrm{Sym}\left(Y^{*}\left(p_{k}\right)\right)$,
contains $\mathrm{Alt}\left(Y^{*}\left(p_{1}\right)\right)\times\ldots\times\mathrm{Alt}\left(Y^{*}\left(p_{k}\right)\right)$.
\end{lem}
\begin{proof}
We prove the proposition by induction on $k$, the case $k=1$ being
trivial. Assume $k\ge2$. It is enough to show that for every $j=1,\ldots,k$,
\begin{equation}
Q_{n}\ge1\times\ldots\times1\times\mathrm{Alt}\left(Y^{*}\left(p_{j}\right)\right)\times1\times\ldots\times1.\label{eq:Q_n >=00003D 1x...x1xAnx1x...x1}
\end{equation}
Recall that $Y^{*}\left(3\right)=\emptyset$, so we may assume $3\nmid n$.
Without loss of generality we assume that $j=k$. We first prove (\ref{eq:Q_n >=00003D 1x...x1xAnx1x...x1})
assuming $p_{k}\ge5$. Note that $\left|Y^{*}\left(p_{k}\right)\right|\ge5$
(see Lemmas \ref{lem:rotations a la BGS p=00003D1(4)} and \ref{lem:rotations a la BGS p=00003D3(4)}),
and so $\mathrm{Alt}\left(Y^{*}\left(p_{k}\right)\right)$ is simple.
This group is never a composition (Jordan-H\"{o}lder) factor of $\mathrm{Alt}\left(Y^{*}\left(p_{\ell}\right)\right)$
when $p_{k}\ne p_{\ell}$, because\footnote{For $p$ odd the size of $\left|Y^{*}\left(p\right)\right|$ is $\frac{p\left(p\pm3\right)}{4}$
as given in Section \ref{sec:Preliminaries}, and $\left|Y^{*}\left(2\right)\right|=4$.} $\left|Y^{*}\left(p_{k}\right)\right|\ne\left|Y^{*}\left(p_{\ell}\right)\right|$.
Now consider the normal series
\begin{equation}
\xymatrix{Q_{n}=Q_{n}\cap\left[\mathrm{Sym}\left(Y^{*}\left(p_{1}\right)\right)\times\ldots\times\mathrm{Sym}\left(Y^{*}\left(p_{k}\right)\right)\right]\\
Q_{n}\cap\left[1\times\ldots\times1\times\mathrm{Sym}\left(Y^{*}\left(p_{k}\right)\right)\right]\ar@{-}[u]^{\trianglelefteq}\\
Q_{n}\cap\left[1\times\ldots\times1\times\mathrm{Alt}\left(Y^{*}\left(p_{k}\right)\right)\right]\ar@{-}[u]^{\trianglelefteq}\\
1\ar@{-}[u]^{\trianglelefteq}
}
\label{eq:normal series 1}
\end{equation}
The group $Q_{p_{k}}$ is a quotient of $Q_{n}$, and so $\mathrm{Alt}\left(Y^{*}\left(p_{k}\right)\right)$
a composition factor of $Q_{n}$, and thus a composition factor of
one of the quotients in (\ref{eq:normal series 1}). But the upper
quotient is isomorphic to $Q_{p_{1}\cdots p_{k-1}}$, which by the
induction hypothesis has composition factors $\mathrm{Alt}\left(Y^{*}\left(p_{\ell}\right)\right)$
for $\ell\ne k,p_{\ell}\ne2$ and possibly some copies of $\nicefrac{\mathbb{Z}}{2\mathbb{Z}}$
coming from $\nicefrac{\mathrm{Sym}\left(Y^{*}\left(p_{\ell}\right)\right)}{\mathrm{Alt}\left(Y^{*}\left(p_{\ell}\right)\right)}$
or copies of $\nicefrac{\mathbb{Z}}{2\mathbb{Z}}$ and $\nicefrac{\mathbb{Z}}{3\mathbb{Z}}$
coming from $\mathrm{Sym}\left(Y^{*}\left(2\right)\right)$. The middle
quotient is either trivial or $\nicefrac{\mathbb{Z}}{2\mathbb{Z}}$.
Thus $\mathrm{Alt}\left(Y^{*}\left(p_{k}\right)\right)$ must be a
composition factor of the bottom quotient, so $1\times\ldots\times1\times\mathrm{Alt}\left(Y^{*}\left(p_{k}\right)\right)\le Q_{n}$.

Finally, if $p_{k}=2$, note that $\left|Y^{*}\left(2\right)\right|=4$.
The composition factors of $\mathrm{Alt}\left(4\right)$ are one copy
of $\nicefrac{\mathbb{Z}}{3\mathbb{Z}}$ and two copies of $\nicefrac{\mathbb{Z}}{2\mathbb{Z}}$.
By an argument as above, the factor $\nicefrac{\mathbb{Z}}{3\mathbb{Z}}$
must belong to the bottom quotient in (\ref{eq:normal series 1}).
Denote 
\[
H\stackrel{\mathrm{def}}{=}Q_{n}\cap\left[1\times\ldots\times1\times\mathrm{Alt}\left(Y^{*}\left(2\right)\right)\right]=1\times\ldots\times1\times H'.
\]
It is easy to check that $H\trianglelefteq Q_{n}$. For every $g_{k}\in\mathrm{Alt}\left(Y^{*}\left(2\right)\right)$
there are $g_{1},\ldots,g_{k-1}$ with $g_{j}\in\mathrm{Sym}\left(Y^{*}\left(p_{j}\right)\right)$
such that $\left(g_{1},\ldots,g_{k}\right)\in Q_{n}$, thus $H'\trianglelefteq\mathrm{Alt}\left(Y^{*}\left(2\right)\right)\cong\mathrm{Alt}\left(4\right)$.
But the only normal subgroup of $\mathrm{Alt}\left(4\right)$ containing
the composition factor $\nicefrac{\mathbb{Z}}{3\mathbb{Z}}$ is $\mathrm{Alt}\left(4\right)$
itself.
\end{proof}

\subsection{Transitivity of $\Gamma$ on $X^{*}\left(n\right)$\label{subsec:Transitivity-on-X*(n)}}

We now finish the proof of Theorem \ref{thm:strong approximation for square-free composite}
and prove that if $n=p_{1}\cdots p_{k}$ is a product of distinct
primes with $Q_{p_{j}}\ge\mathrm{Alt}\left(Y^{*}\left(p_{j}\right)\right)$
for every $1\le j\le k$, then $\Gamma$ acts transitively on $X^{*}\left(n\right)$.

We want the proof of this section to work in a slightly greater generality
than the assumption that $Q_{p_{j}}\ge\mathrm{Alt}\left(Y^{*}\left(p_{j}\right)\right)$,
so that it applies also for the next section, where we do not rely
on the CFSG. This is part of the motivation for the following notation:

\begin{notation} \label{notation:primes,kernels}Let $n=p_{1}\cdots p_{k}$
be a product of distinct primes for which $Q_{p_{j}}$ is primitive.
We assume further that
\begin{itemize}
\item The primes are ordered by the order of the rotations $\rot_{i}$ in
the groups $Q_{p_{j}}$, which is 
\[
\left|\rot_{1}\right|~\mathrm{in}~Q_{p}=\begin{cases}
3 & p=2\\
\frac{p\left(p^{2}-1\right)}{4} & p\equiv1\left(4\right)\\
\frac{p^{2}-1}{4} & p\equiv3\left(4\right)
\end{cases}.
\]
For instance, $7$ comes before $5$. We break potential ties by putting
the larger prime first: for example, we put $11$ before $5$.
\item Without loss of generality, $2,5,7,11\mid n$ and so the first four
primes are $2,7,11,5$ (in that order). This assumption is possible
because in these four cases, computer simulations indicate that $Q_{p}=\mathrm{Sym}\left(Y^{*}\left(p\right)\right)$
is the full symmetric group, so our assumptions always hold.
\end{itemize}
Furthermore, for $j=1,\ldots,k$,
\begin{itemize}
\item Let $M_{j}=p_{1}\cdots p_{j}$\marginpar{$M_{j}$} denote the product
of the first $j$ primes. 
\item Let \marginpar{$\Omega_{j}$}$\Omega_{j}\trianglelefteq\Gamma$ denote
the kernel of the action of $\Gamma$ on $Y^{*}\left(M_{j}\right)$.
Note than $\Omega_{j+1}\trianglelefteq\Omega_{j}$.
\item Let \marginpar{$\Lambda_{j}$}$\Lambda_{j}\trianglelefteq\Gamma$
denote the kernel of the action of $\Gamma$ on $X^{*}\left(M_{j}\right)$.
Note that $\Lambda_{j+1}\trianglelefteq\Lambda_{j}\trianglelefteq\Omega_{j}$.
\end{itemize}
Finally, for every prime $p$, we let $\pi_{p}\colon\Gamma\to Q_{p}$\marginpar{$\pi_{p}$}
denote the projection.

\end{notation} 

In Section \ref{subsec:transitivity on Y*(n) without CFSG} we shall
prove the following lemma without relying on the CFSG:
\begin{lem}
\label{lem:primitive =00003D=00003D> socle contained in kernel}Let
$n=p_{1}\cdots p_{k}$ with $Q_{p_{j}}$ primitive for $j=1,\ldots,k$
as in Notation \ref{notation:primes,kernels}. Then, for every $j=2,\ldots,k$,
the image of $\Omega_{j-1}$ in $Q_{p_{j}}$ contains a subgroup $H_{j}\le\mathrm{Sym}\left(Y^{*}\left(p_{j}\right)\right)$
satisfying:
\begin{enumerate}
\item $H_{j}$ is transitive on $Y^{*}\left(p_{j}\right)$
\item $H_{j}$ is isomorphic to a direct product of non-abelian simple groups\footnote{Note that we assume $j\ge2$. Indeed, this does not hold for $p_{1}=2$:
there are no simple non-abelian subgroups inside $\mathrm{Sym}\left(Y^{*}\left(2\right)\right)\cong\mathrm{Sym}\left(4\right)$.} $T_{1}\times\ldots\times T_{m}$ for some $m=m\left(j\right)\in\mathbb{Z}_{\ge1}$.
\end{enumerate}
In particular, $\Omega_{j-1}$ acts transitively on $Y^{*}\left(p_{j}\right)$
and $\Gamma$ acts transitively on $Y^{*}\left(n\right)$.
\end{lem}
Note that if we assume that $Q_{p_{j}}\ge\mathrm{Alt}\left(Y^{*}\left(p_{j}\right)\right)$,
the conclusion of Lemma \ref{lem:primitive =00003D=00003D> socle contained in kernel}
follows immediately from Lemma \ref{lem:alt =00003D=00003D> transitivity on Y*(n)}:
indeed, for $p\ge5$, $\mathrm{Alt}\left(Y^{*}\left(p\right)\right)$
is indeed transitive on $Y^{*}\left(p\right)$ and is a product of
a single non-abelian simple group. So Lemma \ref{lem:primitive =00003D=00003D> socle contained in kernel}
is already proven relying on the CFSG, or if one assumes that $p_{j}\equiv1\left(4\right)$
for $j=1,\ldots,k$. In the remaining part of this subsection we rely
only on the conclusion of Lemma \ref{lem:primitive =00003D=00003D> socle contained in kernel}.
We assume Notation \ref{notation:primes,kernels} throughout.
\begin{lem}
\label{lem:Lambda transitive on Y*}For $j=2,\ldots,k$, the group
$\Lambda_{j-1}$ acts transitively on $Y^{*}\left(p_{j}\right)$.
\end{lem}
\begin{proof}
Consider the normal series 
\begin{equation}
\xymatrix{\Gamma\\
\Omega_{j-1}\ar@{-}[u]^{\trianglelefteq}\\
\Lambda_{j-1}\ar@{-}[u]^{\trianglelefteq}
}
\label{eq:normal series 2}
\end{equation}
and its projection on $Q_{p_{j}}$ via $\pi_{j}\colon\Gamma\twoheadrightarrow Q_{p_{j}}$.
By Lemma \ref{lem:primitive =00003D=00003D> socle contained in kernel},
$\pi_{j}\left(\Omega_{j-1}\right)\ge H_{j}$ where $H_{j}$ acts transitively
on $Y^{*}\left(p_{j}\right)$ and is a direct product of non-abelian
simple groups. As $\Omega_{j-1}$ fixes $Y^{*}\left(M_{j-1}\right)=Y^{*}\left(p_{1}\right)\times\ldots\times Y^{*}\left(p_{j-1}\right)$,
its action on $X^{*}\left(M_{j-1}\right)$ fixes every $4$-block
and only permutes elements inside the $4$-blocks, hence the image
of $\Omega_{j-1}$ in $\Gamma_{M_{j-1}}$ is a subgroup of $\mathrm{Sym}\left(4\right)^{\left|Y^{*}\left(p_{1}\right)\right|+\ldots+\left|Y^{*}\left(p_{j-1}\right)\right|}$.
Hence this image is solvable of order $2^{\alpha}\cdot3^{\beta}$
for some $\alpha,\beta\in\mathbb{Z}_{\ge0}$, so all its composition
factors are either $\nicefrac{\mathbb{Z}}{2\mathbb{Z}}$ or $\nicefrac{\mathbb{Z}}{3Z}$.
We deduce that the quotient $\nicefrac{\Omega_{j-1}}{\Lambda_{j-1}}$
has only composition factors $\nicefrac{\mathbb{Z}}{2\mathbb{Z}}$
and\textbackslash{}or $\nicefrac{\mathbb{Z}}{3\mathbb{Z}}$. Let 
\[
\Omega_{j-1}=N_{0}\trianglerighteq N_{1}\trianglerighteq\ldots\trianglerighteq N_{r}=\Lambda_{j-1}
\]
be a normal series with quotients $\nicefrac{\mathbb{Z}}{2\mathbb{Z}}$
and\textbackslash{}or $\nicefrac{\mathbb{Z}}{3\mathbb{Z}}$. Note
that the index $\left[H_{j}\colon\pi_{j}\left(N_{1}\right)\cap H_{j}\right]$
is at most $3$, but as $H_{j}$ is a direct product of non-abelian
simple groups, it has no proper subgroups of index\footnote{To be sure, the reason that $H=T_{1}\times\ldots\times T_{m}$ with
$T_{1},\ldots,T_{m}$ finite non-abelian simple groups has no subgroups
of index $2$ or $3$ is that the normal subgroups of $H$ are $B_{1}\times\ldots\times B_{m}$
where $B_{i}\in\left\{ 1,T_{i}\right\} $ for every $i$ (this is
standard: if $N\trianglelefteq H$ and $N\cap T_{1}\ne1$, then $1\ne\left[N,T_{1}\right]\trianglelefteq T_{1}$,
and so $\left[N,T_{1}\right]=T_{1}$ and $N\ge T_{1}$). In particular,
since the smallest non-abelian simple group is $\mathrm{Alt}\left(5\right)$,
any proper normal subgroup of $H$ is of index at least $60$. If
$K\le H$ has index $2$ or $3$, then its core, $\cap_{h\in H}hKh^{-1}$,
is proper normal subgroup of index at most $6$, which is impossible.} $\le3$, hence $\pi_{j}\left(N_{1}\right)\ge H_{j}$. By induction,
the same argument shows that $\pi_{j}\left(N_{\ell}\right)\ge H_{j}$
for every $\ell$, and, in particular, $\pi_{j}\left(\Lambda_{j-1}\right)\ge H_{j}$.
\end{proof}
\begin{lem}
\label{lem:Lambda transitive on X*(p_k)}For $j=5,\ldots,k$ (so $p_{j}\ge13$),
$\Lambda_{j-1}$ acts transitively on $X^{*}\left(p_{j}\right)$.
\end{lem}
\begin{proof}
Our strategy is to find a triple $\left(x,y,z\right)\in X^{*}\left(p_{j}\right)$
and elements in $\Lambda_{j-1}$ mapping $\left(x,y,z\right)$ to
the other elements in its $4$-block: $\left(x,-y,-z\right)$, $\left(-x,y,-z\right)$
and $\left(-x,-y,z\right)$. Together with the transitivity of $\Lambda_{j-1}$
on $Y^{*}\left(p_{k}\right)$ established in Lemma \ref{lem:Lambda transitive on Y*},
this would complete the proof. As in other places in this paper, we
deal separately with the case $p_{j}\equiv1\left(4\right)$ and the
case $p_{j}\equiv3\left(4\right)$, the argument in the former case
being simpler.\\

\noindent \textbf{Case 1: $p=p_{j}\equiv1\left(4\right)$}

\noindent Take some $x\in\mathbb{F}_{p}$ hyperbolic of maximal order
(namely, the $\rot_{1}$-cycles in $C_{1}\left(x\right)$ are of length
$p-1\ge12$ each). Since $0$ has order $4$, $x\ne0$ and $\left(0,x,ix\right)\in X^{*}\left(p\right)$.
Let $\left(r,s,t\right)\in X^{*}\left(p\right)$ be another solution
with $r$ elliptic. As all $\rot_{1}$-cycles in $C_{1}\left(0\right)$
have length 4 and $\left(p+1\equiv2\mod4\right)$, we get that $\rot_{1}^{~p+1}$
fixes all four elements in $\left[r,s,t\right]$ while mapping $\left(0,x,ix\right)\mapsto\left(0,-x,-ix\right)$.
By Lemma \ref{lem:Lambda transitive on Y*}, there is some $g\in\Lambda_{j-1}$
mapping $\left[0,x,ix\right]\mapsto\left[r,s,t\right]$. The element
$h_{1}=g^{-1}\cdot\rot_{1}^{~-\left(p+1\right)}\cdot g\cdot\rot_{1}^{~p+1}$
is in $\Lambda_{j-1}$ (as $\Lambda_{j-1}\trianglelefteq\Gamma$)
and maps $\left(0,x,ix\right)\mapsto\left(0,-x,-ix\right)$. 

Since $x$ is maximal hyperbolic, its order is $\left(p-1\right)$
which is divisible by $4$. Hence $-x$ is also maximal hyperbolic.
Let now $\left(r',s',t'\right)\in X^{*}\left(p\right)$ be a solution
with $s'$ elliptic. Note that $\left(\frac{p^{2}-1}{4}\equiv0\mod p+1\right)$
while $\left(\frac{p^{2}-1}{4}\equiv\frac{p-1}{2}\mod p-1\right)$.
Thus $\rot_{2}^{~\left(p^{2}-1\right)/4}$ fixes all four elements
in $\left[r',s',t'\right]$ while mapping $\left(0,x,ix\right)\mapsto\left(0,x,-ix\right)$
and $\left(0,-x,-ix\right)\mapsto\left(0,-x,ix\right)$. By Lemma
\ref{lem:Lambda transitive on Y*}, there is some $g'\in\Lambda_{j-1}$
mapping $\left[0,x,ix\right]\mapsto\left[r',s',t'\right]$. The element
$h_{1}=\left(g'\right)^{-1}\cdot\rot_{1}^{~-\left(p^{2}-1\right)/4}\cdot g'\cdot\rot_{1}^{~\left(p^{2}-1\right)/4}$
is in $\Lambda_{j-1}$ and maps $\left(0,x,ix\right)\mapsto\left(0,x,-ix\right)$
and $\left(0,-x,-ix\right)\mapsto\left(0,-x,ix\right)$.\\

\noindent \textbf{Case 2: $p=p_{j}\equiv3\left(4\right)$}

\noindent In Proposition \ref{prop:(e,e,*) with order divisible by 4}
below, we prove there is a solution $\left(x,y,z\right)\in X^{*}\left(p\right)$
with both $x$ and $y$ elliptic of order divisible by $4$. In this
case, $-x$ has the same order as $x$, say this order is $4m$ and
note that $4m|\left(p+1\right)$. Let $\left(r,s,t\right)\in X^{*}\left(p\right)$
be another solution with $r$ hyperbolic. As $p-1\equiv2\left(4\right)$,
there is a number $q$ with $\left(q\equiv2m\mod4m\right)$ and $\left(q\equiv0\mod p-1\right)$.
We get that $\rot_{1}^{~q}$ fixes all four elements in $\left[r,s,t\right]$
while mapping $\left(x,y,z\right)\mapsto\left(x,-y,-z\right)$ and
$\left(-x,-y,z\right)\mapsto\left(-x,y,-z\right)$. By Lemma \ref{lem:Lambda transitive on Y*},
there is some $g\in\Lambda_{j-1}$ mapping $\left[x,y,z\right]\mapsto\left[r,s,t\right]$.
The element $h_{1}=g^{-1}\cdot\rot_{1}^{~-q}\cdot g\cdot\rot_{1}^{~q}$
is in $\Lambda_{j-1}$ and maps $\left(x,y,z\right)\mapsto\left(x,-y,-z\right)$
and $\left(-x,-y,z\right)\mapsto\left(-x,y,-z\right)$. In the same
fashion, we find an element of $\Lambda_{j-1}$ mapping $\left(x,y,z\right)\mapsto\left(-x,y,-z\right)$
and we are done. 
\end{proof}
Modulo Proposition \ref{prop:(e,e,*) with order divisible by 4} which
we prove at the end of this subsection, we can now complete the proofs
of Theorem \ref{thm:strong approximation for square-free composite}
and Corollary \ref{cor:transitive for composite with factors we handle}:
\begin{proof}[Proof of Theorem \ref{thm:strong approximation for square-free composite}]
 We use Notation \ref{notation:primes,kernels}. We need to show
that $\Gamma$ acts transitively on $X^{*}\left(n\right)$. We prove
that $\Gamma$ acts transitively on $X^{*}\left(M_{j}\right)$ for
$j=1,\ldots,k$ (recall that $M_{k}=n$). For $j=4$ we verified by
computer that $\Gamma$ is transitive on $X^{*}\left(2\cdot5\cdot7\cdot11\right)$.
For $j\ge5$, we use induction and assume that $\Gamma$ acts transitively
on $X^{*}\left(M_{j-1}\right)$. From Lemma \ref{lem:Lambda transitive on X*(p_k)}
it follows that $\Gamma$ is transitive on $X^{*}\left(M_{j}\right)$. 
\end{proof}
\medskip{}
We complete the subsection with the proposition we use in the proof
of case 2 in Lemma \ref{lem:Lambda transitive on X*(p_k)}:
\begin{prop}
\label{prop:(e,e,*) with order divisible by 4}For every prime $p\ne3,11$
with $p\equiv3\left(4\right)$, there is a solution $\left(x,y,z\right)\in X^{*}\left(p\right)$
with two coordinates elliptic of order divisible by 4.
\end{prop}
In the proof of Proposition \ref{prop:(e,e,*) with order divisible by 4}
we use notation as in Section \ref{subsec:No-correlation-ellip}.
As $4|\left(p+1\right)$, if $\omega\in H$ is not a square then $4|\left|\omega\right|$.
Thus, it is enough to find a solution $\left(x,y,z\right)\in X^{*}\left(p\right)$
with $x,y$ elliptic and the corresponding $\omega_{x},\omega_{y}$
not squares in $H$.
\begin{lem}
\label{lem:y+2 square for y elliptic-1}Assume $y=\omega+\omega^{-1}$
is elliptic (so $\omega\in H$). Then $\omega$ is a square in $H$
if and only if $y+2$ is a square in $\mathbb{F}_{p}$.
\end{lem}
\begin{proof}
Note that $y+2=\omega+\omega^{-1}+2=\left(\omega^{1/2}+\omega^{-1/2}\right)^{2}$.
If $\omega^{1/2}\in H$ then $\omega^{1/2}+\omega^{-1/2}\in\mathbb{F}_{p}$.
On the other hand, if $\omega^{1/2}\notin H$, then $\omega^{\left(p+1\right)/2}=-1$
and so $\omega^{1/2}+\omega^{-1/2}\notin\mathbb{F}_{p}$, because
\[
\left(\omega^{1/2}+\omega^{-1/2}\right)^{p}=\omega^{\left(p+1\right)/2}\omega^{-1/2}+\omega^{-\left(p+1\right)/2}\omega^{1/2}=-\left(\omega^{-1/2}+\omega^{1/2}\right)\ne\left(\omega^{1/2}+\omega^{-1/2}\right)
\]
(the last inequality stems from $\left(\omega^{1/2}+\omega^{-1/2}\right)^{2}=y+2\ne0$). 
\end{proof}

\begin{proof}[Proof of Proposition \ref{prop:(e,e,*) with order divisible by 4}]
 Fix $x\in\mathbb{F}_{p}$ elliptic of maximal order ($p+1$). So
$4|\left|\omega_{x}\right|=p+1$. By Lemma \ref{lem:y+2 square for y elliptic-1},
it is enough to find $y,z\in\mathbb{F}_{p}$ such that $\left(x,y,z\right)\in X^{*}\left(p\right)$
is a solution, $y$ is elliptic and $y+2$ is a non-square. Since
$y$ elliptic means that $y^{2}-4=\left(y+2\right)\left(y-2\right)$
is not a square, we need to find $y,z$ with $\left(x,y,z\right)\in X^{*}\left(p\right)$
and $y+2$ a non-square and $y-2$ a square.

Imitating the notation from Section \ref{subsec:No-correlation-ellip},
assume $x=\omega+\omega^{-1}$ with $\omega\in H$, choose some $A\in\mathbb{F}_{p^{2}}$
for which $A^{p+1}=\frac{x^{2}}{x^{2}-4}$, and let $f_{A}\left(h\right)=Ah+A^{p}h^{-1}$
for $h\in H$. Then, 
\begin{equation}
\left\{ \left(f_{A}\left(h\right),f_{A\omega}\left(h\right)\right)\,\middle|\,h\in H\right\} =\left\{ \left(y,z\right)\,\middle|\,\left(x,y,z\right)\in X^{*}\left(p\right)\right\} .\label{eq:from (y,z) to f(h)-1}
\end{equation}
Recall the parametrization of $H\setminus\left\{ -i\right\} $ by
elements from $\mathbb{F}_{p}$ described in Lemma \ref{lem:H}: $h\left(s\right)=\frac{2s+i\left(1-s^{2}\right)}{1+s^{2}}=\frac{-i\left(s+i\right)}{s-i}$.
Define $g_{1},g_{2}\in\mathbb{F}_{p}\left[s\right]$ as follows:
\begin{eqnarray*}
g_{1}\left(s\right) & \stackrel{\mathrm{def}}{=} & \left(1+s^{2}\right)^{2}\left[f_{A}\left(h\left(s\right)\right)+2\right]=\left(1+s^{2}\right)\left[2s\left(A+A^{p}\right)+\left(1-s^{2}\right)i\left(A-A^{p}\right)+2\left(1+s^{2}\right)\right]\\
g_{2}\left(s\right) & \stackrel{\mathrm{def}}{=} & \left(1+s^{2}\right)^{2}\left[f_{A}\left(h\left(s\right)\right)-2\right]=\left(1+s^{2}\right)\left[2s\left(A+A^{p}\right)+\left(1-s^{2}\right)i\left(A-A^{p}\right)-2\left(1+s^{2}\right)\right].
\end{eqnarray*}
It is not hard to see that $g_{j}\left(s\right)\in\mathbb{F}_{p}\left[s\right]$:
indeed, $A+A^{p},i\left(A-A^{p}\right)\in\mathbb{F}_{p}$. We now
show that for large enough $p$, there is some $s\in\mathbb{F}_{p}$
for which 
\begin{equation}
\left(\frac{g_{1}\left(s\right)}{p}\right)=-1~~\mathrm{and}~~\left(\frac{g_{2}\left(s\right)}{p}\right)=1.\label{eq:-1,1}
\end{equation}
Denote by $N_{\left(-1,1\right)}$ the number of $s\in\mathbb{F}_{p}$
for which (\ref{eq:-1,1}) holds. Our goal is to show that for large
enough $p$, $N_{\left(-1,1\right)}>0$. As in the proof of Proposition
\ref{prop:high-order =00003D=00003D> no self-correlation}, $g_{1}$
and $g_{2}$ have no zeros inside $\mathbb{F}_{p}$ because there
are no solutions in $X^{*}\left(p\right)$ involving $\pm2$. So
\begin{equation}
N_{(-1,1)}=\frac{1}{4}\sum_{s\in\mathbb{F}_{p}}\left(1-\left(\frac{g_{1}\left(s\right)}{p}\right)\right)\left(1+\left(\frac{g_{2}\left(s\right)}{p}\right)\right).\label{eq:long formula for n(-1,1)}
\end{equation}
For $\emptyset\ne B\subseteq\left\{ 1,2\right\} $, let $M_{B}\stackrel{\mathrm{def}}{=}\sum_{s\in\mathbb{F}_{p}}\left(\frac{\prod_{j\in B}g_{j}\left(s\right)}{p}\right)$
and then (\ref{eq:long formula for n(-1,1)}) becomes
\begin{equation}
N_{\left(-1,1\right)}=\frac{1}{4}\left(p-M_{\left\{ 1\right\} }+M_{\left\{ 2\right\} }-M_{\left\{ 1,2\right\} }\right).\label{eq:formula for (-1,1)}
\end{equation}
Note that 
\[
g_{1}\left(s\right)g_{2}\left(s\right)=\left(1+s^{2}\right)^{4}\left[f_{A}\left(h\left(s\right)\right)^{2}-4\right]=\left(1+s^{2}\right)^{2}g_{A}\left(s\right)
\]
where $g_{A}\left(s\right)$ is defined as in Equation (\ref{eq:g in terms of h})
in Section \ref{subsec:No-correlation-ellip} for $m=1$. As our analysis
in Section \ref{subsec:No-correlation-ellip} shows, all roots of
$g_{A}$, except for $\pm i$, have multiplicity $1$. Thus, none
of $g_{1}$, $g_{2}$ or $g_{1}g_{2}$ is a square in $\overline{\mathbb{F}_{p}}\left[x\right]$.
Now $g_{1}$ and $g_{2}$ have each at most 4 distinct roots and by
Theorem \ref{thm:Weil}, $\left|M_{\left\{ 1\right\} }\right|,\left|M_{\left\lfloor 2\right\rfloor }\right|\le3\sqrt{p}$.
Their product $g_{1}g_{2}$ has at most $6$ distinct roots, hence
by Theorem \ref{thm:Weil} $\left|M_{\left\{ 1,2\right\} }\right|\le5\sqrt{p}$.
From (\ref{eq:formula for (-1,1)}) we get
\begin{eqnarray*}
N_{\left(-1,1\right)} & \ge & \frac{1}{4}\left(p-2\cdot3\sqrt{p}-5\sqrt{p}\right)=\frac{p-11\sqrt{p}}{4}.
\end{eqnarray*}
So for $p>11^{2}=121$ we have $N_{\left(-1,1\right)}>0$ and we are
done.\\
For all primes $p$ with $p\equiv3\left(4\right)$, $p\le121$ and
$p\ne3,11$, we verified by a computer there is a solution $\left(x,y,z\right)\in X^{*}\left(p\right)$
with $x,y$ elliptic and of order divisible by $4$. For example,
one can take $\left(3,3,3\right)\in X^{*}\left(7\right)$, $\left(6,6,8\right)\in X^{*}\left(19\right)$,
$\left(3,3,3\right)\in X^{*}\left(23\right)$ and $\left(4,4,9\right)\in X^{*}\left(31\right)$.
\end{proof}

\subsection{Transitivity without the classification\label{subsec:transitivity on Y*(n) without CFSG}}

In this section we prove Theorem \ref{thm:if primitive then transitive on product}
concerning the transitivity of $\Gamma$ in square free composite
moduli without relying on the CFSG. We are going to use some strong
results from the theory of permutation groups, mostly revolving around
O'Nan-Scott theorem. While strong, the proofs of these results are
completely contained in the book \cite{dixon1996permutation} and
are not more than a few-page-long each. We stress that if all primes
in the decomposition of $n$ are $2$ or $\left(1\mod4\right)$, then
already the proof in the previous sections does not rely on the CFSG. 

More concretely, let $n=p_{1}\cdots p_{k}$ be a product of distinct
primes, and we assume that $Q_{p_{j}}$ is a primitive permutation
group in its action on $Y^{*}\left(p_{j}\right)$ for every $j=1,\ldots,k$.
Our goal is to show then that $\Gamma$ acts transitively on $X^{*}\left(n\right)$.
It is enough to prove Lemma \ref{lem:primitive =00003D=00003D> socle contained in kernel}
above, as we already showed in Section \ref{subsec:Transitivity-on-X*(n)}
how it yields the conclusion we seek. Throughout this subsection we
assume Notation \ref{notation:primes,kernels}.

The CFSG-free proof of Lemma \ref{lem:primitive =00003D=00003D> socle contained in kernel}
uses the important concept of the socle:
\begin{defn}
\label{def:socle}A minimal normal subgroup of a non-trivial group
$G$ is a normal subgroup $K\ne1$ of $G$ which does not contain
properly any other non-trivial normal subgroup of $G$. The \textbf{socle}
of $G$, denoted \emph{}\marginpar{\emph{$\protect\soc\left(G\right)$}}$\soc\left(G\right)$,
is the subgroup generated by the set of all minimal normal subgroups
of $G$. Note that $\soc\left(G\right)$ is generated by normal subgroups
of $G$ and thus $\soc\left(G\right)\trianglelefteqslant G$.
\end{defn}
For example, if $m\ge5$ then $\soc\left(\mathrm{Sym}\left(m\right)\right)=\soc\left(\mathrm{Alt}\left(m\right)\right)=\mathrm{Alt}\left(m\right)$.
In contrast, $\soc\left(\mathrm{Sym}\left(4\right)\right)=\soc\left(\mathrm{Alt}\left(4\right)\right)=\left\{ 1,\left(12\right)\left(34\right),\left(13\right)\left(24\right),\left(14\right)\left(23\right)\right\} $. 
\begin{thm}[{See \cite[Theorems 4.3B, Corollary 4.3B and Theorem 4.7A]{dixon1996permutation}}]
\label{thm:socle of primitive} Let $G\le\mathrm{Sym}\left(n\right)$
be a primitive subgroup. Then exactly one of the following holds:
\begin{enumerate}
\item For some prime $p$ and some integer $d$, the group $G$ is permutation
isomorphic\footnote{Two permutation groups are permutation isomorphic if they are the
same permutation groups except for, possibly, the labeling of the
points in the sets they act on.} to a subgroup of the affine group $\mathrm{Aff}\left(p,d\right)$
acting on $\mathbb{F}_{p}^{~d}$, so, in particular, $n=p^{d}$. In
this case, $\soc\left(G\right)$ is a regular\footnote{A permutation group $H\le\mathrm{Sym}\left(n\right)$ is called \emph{regular}
if it is sharply transitive. Namely, it is transitive and free. In
other words, it is transitive and of order $n$. The name originates
from the observation that such subgroups are obtained as the (left
or right) regular representation of order-$n$ groups.} elementary abelian subgroup of order $p^{d}$.
\item $\soc\left(G\right)=K_{1}\times K_{2}$ where $K_{1},K_{2}\trianglelefteqslant G$
are minimal normal subgroups of $G$, which are regular, non-abelian
and permutation isomorphic to each other. Moreover\footnote{For $G$ a group and $K\le G$ a subgroup, $C_{G}\left(K\right)=\left\{ g\in G\,\middle|\,gk=kg~\forall k\in K\right\} $
is the centralizer of $K$ in $G$.}, $C_{G}\left(K_{1}\right)=K_{2}$ and $C_{G}\left(K_{2}\right)=K_{1}$.
In addition, $K_{1}\cong K_{2}\cong T^{m}$ for some finite simple
non-abelian group $T$ and some $m\in\mathbb{Z}_{\ge1}$.
\item $\soc\left(G\right)$ is a minimal normal subgroup of $G$. Moreover,
$C_{G}\left(\soc\left(G\right)\right)=1$ and $\soc\left(G\right)\cong T^{m}$
for some finite simple non-abelian group $T$ and some $m\in\mathbb{Z}_{\ge1}$.
\end{enumerate}
\end{thm}

\begin{thm}[{See \cite[Theorem 1.6A]{dixon1996permutation}}]
\label{thm:normal subgp of primitive is transitive} If $G\le\mathrm{Sym}\left(n\right)$
is a primitive permutation group and $1\ne H\trianglelefteqslant G$
is a non-trivial normal subgroup, then $H$ is transitive.
\end{thm}
\begin{cor}
\label{cor:socle is a product of NAFSGs}If $p\ge5$ is prime and
$Q_{p}$ is primitive, then \marginpar{$\protect\soc\left(p\right)$}$\soc\left(p\right)\stackrel{\mathrm{def}}{=}\soc\left(Q_{p}\right)$
acts transitively on $Y^{*}\left(p\right)$ and is a direct product
of non-abelian simple groups.
\end{cor}
\begin{proof}
Transitivity follows from Theorem \ref{thm:normal subgp of primitive is transitive}
and the fact that the socle is a normal subgroup. Case (1) of Theorem
\ref{thm:socle of primitive} is ruled out because $\left|Y^{*}\left(p\right)\right|=\frac{p\left(p\pm3\right)}{4}$
is not a prime power (or, alternatively, because $\mathrm{Aff}\left(p,d\right)$
has no non-identity elements fixing more than half of the points,
such as $\rot_{1}^{p\left(p+1\right)/2}\in Q_{p}$). So either $Q_{p}$
falls into case (2) or it falls into case (3).
\end{proof}
We also use the following result giving strong limitations on primitive
groups:
\begin{thm}[{See \cite[Theorems 5.3A and 5.5B]{dixon1996permutation}}]
\label{thm:bounds on permutation groups} Let $G\lvertneqq\mathrm{Sym}\left(n\right)$,
$G\ne\mathrm{Alt}\left(n\right)$, be a primitive permutation group. 
\begin{enumerate}
\item If $G$ is not $2$-transitive then $\left|G\right|<\exp\left\{ 4\sqrt{n}\left(\ln n\right)^{2}\right\} $.
\item If $n\ge216$ and $G$ is $2$-transitive and contains a section\footnote{A \emph{section} of a group is some quotient of a subgroup.}
isomorphic to $\mathrm{Alt}\left(k\right)$, then $k<6\ln n$.
\end{enumerate}
\end{thm}
\begin{lem}
\label{lem:Q_p primitive than socle contained in image of kernel of Q_q}Let
$p$ and $q$ be distinct primes with $Q_{p}$ and $Q_{q}$ primitive,
and such that $p$ precedes $q$ in the order defined in Notation
\ref{notation:primes,kernels}. Then $Q_{pq}\ge1\times\soc\left(q\right)$
(sitting inside $\mathrm{Sym}\left(Y^{*}\left(p\right)\right)\times\mathrm{Sym}\left(Y^{*}\left(q\right)\right)$).
\end{lem}
\begin{proof}
Recall that the primes are sorted by the order of rotation elements.
So if $o_{p}$ ($o_{q}$, respectively) is the order of $\rot_{1}$
in $Q_{p}$ ($Q_{q}$, respectively) then $o_{p}\le o_{q}$. \\

\noindent \textbf{Case 1: $o_{p}<o_{q}$}

\noindent If the inequality is strict, then the image of $g=\rot_{1}^{o_{p}}\in\Gamma$
in $Q_{p}$ is the identity whereas its image $\overline{g}$ in $Q_{q}$
is not. By Corollary \ref{cor:socle is a product of NAFSGs}, $\soc\left(q\right)$
falls under one of cases $\left(2\right)$ or $\left(3\right)$ from
Theorem \ref{thm:socle of primitive}.

Assume first that $\soc\left(q\right)$ falls under case (3). Since
$C_{Q_{p}}\left(\soc\left(q\right)\right)=1$, there is some $h\in\soc\left(q\right)$
not commuting with $\overline{g}\in Q_{q}$, so $e\ne\left[\overline{g},h\right]=\overline{g}h\overline{g}^{-1}h^{-1}\in\soc\left(q\right)\cap\pi_{q}\left(\ker\left(\Gamma\twoheadrightarrow Q_{p}\right)\right)$.
Since $\soc\left(q\right)$ is a minimal normal subgroup of $Q_{q}$,
it is generated by the conjugates of $\left[\overline{g},h\right]$
in $Q_{q}$, all of which also belong to $\pi_{q}\left(\ker\left(\Gamma\twoheadrightarrow Q_{p}\right)\right)$.
Thus $\soc\left(q\right)\le\pi_{q}\left(\ker\left(\Gamma\twoheadrightarrow Q_{p}\right)\right)$.

Now assume that $\soc\left(q\right)$ falls under case (2). Since
regular subgroups of $\mathrm{Sym}\left(n\right)$ are obtained as
the (left or right) regular representation of a group of order $n$,
every element of a regular permutation group has all its cycles with
equal length. Since $\rot_{1}\in Q_{q}$ contains cycles of coprime
lengths, no non-trivial power of it can belong to a regular subgroup,
so $\overline{g}=\rot_{1}^{o_{p}}\notin K_{1}\cup K_{2}$. So there
are $h_{1}\in K_{1}$ and $h_{2}\in K_{2}$ not commuting with $g$.
Consider $h=h_{1}h_{2}\in K_{1}\times K_{2}=\soc\left(q\right)$.
Then $\left[\overline{g},h\right]=\left(\left[\overline{g},h_{1}\right],\left[\overline{g},h_{1}\right]\right)\in K_{1}\times K_{2}=\soc\left(q\right)$
belongs also to $\pi_{q}\left(\ker\left(\Gamma\twoheadrightarrow Q_{p}\right)\right)$
but not to $K_{1}\cup K_{2}$. The only normal subgroups of $Q_{p}$
which are contained in $K_{1}\times K_{2}$ are $1,K_{1},K_{2}$ and
$K_{1}\times K_{2}$. Hence $K_{1}\times K_{2}$ is generated by the
conjugates in $Q_{q}$ of $\left[\overline{g},h\right]$, all of which
belong to $\pi_{q}\left(\ker\left(\Gamma\twoheadrightarrow Q_{p}\right)\right)$.
Thus $\soc\left(q\right)\le\pi_{p}\left(\ker\left(\Gamma\twoheadrightarrow Q_{p}\right)\right)$.\\

\noindent \textbf{Case 2: $o_{p}=o_{q}$}

\noindent We are left with the rare case\footnote{In fact, the only such case with $p<1{,}000{,}000$ is $p=5$.}
that $o_{p}=o_{q}$, as in $p=11$ and $q=5$. In this case $p>q$,
$p\equiv3\left(4\right)$, $q\equiv1\left(4\right)$ and $\left(p^{2}-1\right)=q\left(q^{2}-1\right)$.
In particular, as $Q_{q}$ is primitive, it contains the full alternating
group $\mathrm{Alt}\left(Y^{*}\left(q\right)\right)$ by the CFSG-free
Theorem \ref{thm:1 mod 4}. We claim that $Q_{p}$ has no composition
factor isomorphic to $\mathrm{Alt}\left(Y^{*}\left(q\right)\right)$.
Using this, we can finish as in the proof of Lemma \ref{lem:alt =00003D=00003D> transitivity on Y*(n)}:
indeed, consider the following normal series of $Q_{pq}$
\begin{equation}
1\trianglelefteq Q_{pq}\cap\left[1\times\mathrm{Alt}\left(Y^{*}\left(q\right)\right)\right]\trianglelefteq Q_{pq}\cap\left[1\times\mathrm{Sym}\left(Y^{*}\left(q\right)\right)\right]\trianglelefteq Q_{pq}.\label{eq:normal series of Q_pq}
\end{equation}
Since $Q_{q}$ is a quotient of $Q_{pq}$, $\mathrm{Alt\left(Y^{*}\left(q\right)\right)}$
is a composition factor of $Q_{pq}$, so it has to be a composition
factor of one of the quotients in (\ref{eq:normal series of Q_pq}).
The rightmost quotient is $Q_{p}$ which we show below has no composition
factor isomorphic to $\mathrm{Alt\left(Y^{*}\left(q\right)\right)}$.
The second quotient is $\nicefrac{\mathbb{Z}}{2\mathbb{Z}}$ or trivial.
Thus, the leftmost quotient contains $\mathrm{Alt\left(Y^{*}\left(q\right)\right)}$
as a composition factor, namely, $Q_{pq}\ge1\times\mathrm{Alt\left(Y^{*}\left(q\right)\right)}$,
and we are done as $\soc\left(q\right)=\mathrm{Alt\left(Y^{*}\left(q\right)\right)}$.

So we have left to show that $Q_{p}$ has no composition factor isomorphic
to $\mathrm{Alt\left(Y^{*}\left(q\right)\right)}$. This is certainly
the case if $Q_{p}\ge\mathrm{Alt\left(Y^{*}\left(p\right)\right)}$
(as in the case $p=11,q=5$). So assume $Q_{p}\ngeqslant\mathrm{Alt}\left(Y^{*}\left(p\right)\right)$
and proceed using Theorem \ref{thm:bounds on permutation groups}. 

First, assume that $Q_{p}$ is \emph{not} $2$-transitive. Asymptotically,
its order is smaller than that of $\mathrm{Alt}\left(Y^{*}\left(q\right)\right)$:
indeed, if $n=\left|Y^{*}\left(p\right)\right|=\frac{p\left(p-3\right)}{4}$
then $n\approx q^{3}$, and so by Theorem \ref{thm:bounds on permutation groups},
\begin{eqnarray*}
\ln\left|\mathrm{Alt}\left(Y^{*}\left(q\right)\right)\right| & = & \ln\left(\frac{1}{2}\cdot\left(\frac{q\left(q+3\right)}{4}\right)!\right)\approx q^{2}\ln q\\
\ln\left|Q_{p}\right| & \le & 4\sqrt{n}\left(\ln n\right)^{2}\approx q^{1.5}\left(\ln q\right)^{2}.
\end{eqnarray*}
In fact, this asymptotic reasoning starts taking effect for $q\ge203{,}897$,
but for smaller values of $q$ there are no cases for which $o_{p}=o_{q}$
except for $q=5$ (this was easily verified by computer). 

Finally, assume that $Q_{p}$ is $2$-transitive. Then, not only does
it not have a composition factor isomorphic to $\mathrm{Alt\left(Y^{*}\left(q\right)\right)}$,
it does not even have a section isomorphic to it: since $\frac{p\left(p-3\right)}{4}\ge216$,
Theorem \ref{thm:bounds on permutation groups} says that $k=\frac{q\left(q+3\right)}{4}<6\ln\frac{p\left(p-3\right)}{4}$.
This is impossible when $q\ge13$. 
\end{proof}
We can now finish our CFSG-free proof of Lemma \ref{lem:primitive =00003D=00003D> socle contained in kernel}.
\begin{proof}[CFSG-free proof of Lemma \ref{lem:primitive =00003D=00003D> socle contained in kernel}]
 Assume $n=p_{1}\cdots p_{k}$ is a product of distinct primes with
$Q_{p_{1}},\ldots,Q_{p_{k}}$ primitive and $p_{1},\ldots,p_{k}$
ordered as in Notation \ref{notation:primes,kernels}. We need to
show that for every $j=2,\ldots,k$, the image of $\Omega_{j-1}$
in $Q_{p_{j}}$, $\pi_{p_{j}}\left(\Omega_{j-1}\right)$ contains
a subgroup $H_{j}\le\mathrm{Sym}\left(Y^{*}\left(p_{j}\right)\right)$
which is transitive and isomorphic to a direct product of non-abelian
simple groups. We show that $\pi_{p_{j}}\left(\Omega_{j-1}\right)\ge\soc\left(p_{j}\right)$,
which is enough by Corollary \ref{cor:socle is a product of NAFSGs}.

Without loss of generality, it is enough to prove this when $j=k$.
As $\soc\left(p_{k}\right)\cong\prod_{i=1}^{m}T_{i}$ with $T_{1},\ldots,T_{m}$
non-abelian simple groups, each of them satisfies $\left[T_{i},T_{i}\right]=T_{i}$.
Hence for any $t\in\mathbb{Z}_{\ge1}$ there is a sequence of elements
$\overline{g}_{1},\ldots,\overline{g}_{t}\in\soc\left(p_{k}\right)$
so that the nested commutator 
\[
\left[\ldots\left[\left[\overline{g}_{1},\overline{g}_{2}\right],\overline{g}_{3}\right],\ldots,\overline{g}_{t}\right]
\]
has non-trivial projection in each of the $T_{i}$'s. Choose such
a sequence of length $t=k-1$. By Lemma \ref{lem:Q_p primitive than socle contained in image of kernel of Q_q},
for every $i=1,\ldots,k-1$, there is an element $g_{i}\in\Gamma$
with $\pi_{p_{i}}\left(g_{i}\right)=1$ and $\pi_{p_{k}}\left(g_{i}\right)=\overline{g}_{i}$.
The element 
\[
g=\left[\ldots\left[\left[g_{1},g_{2}\right],g_{3}\right],\ldots,g_{k-1}\right]\in\Gamma
\]
satisfies then that $\pi_{p_{i}}\left(g\right)=1$ for all $i=1,\ldots,k-1$,
whereas $\pi_{p_{k}}\left(g\right)\in\soc\left(p_{k}\right)$ in not
contained in any proper normal subgroup of $\soc\left(p_{k}\right)$.
Hence every element of $\soc\left(p_{k}\right)$ is a product of conjugates
of $\pi_{p_{k}}\left(g\right)$, and we obtain that $\pi_{p_{k}}\left(\Omega_{k-1}\right)\ge\soc\left(p_{k}\right)$.
\end{proof}

\section{$T_{2}$-systems\label{sec:T-systems}}

This section explains why Theorem \ref{thm:Action on T2-system} is
equivalent to Theorems \ref{thm:1 mod 4} and \ref{thm:3 mod 4: alternating given specific conditions}.
Namely, if we let $\Sigma_{2,-2}\left(p\right)$ denote the set of
$\mathrm{PSL}\left(2,p\right)$-defining subgroups of $\F_{2}$ with
associated trace $-2$, our goal here is to show:
\begin{enumerate}
\item A one-to-one correspondence between $Y^{*}\left(p\right)$ and $\Sigma_{2,-2}\left(p\right)$,
and
\item An isomorphism between $Q_{p}$, the group of permutations induced
by the action of $\Gamma$ on $Y^{*}\left(p\right)$, and the group
of permutations induced by the action of $\mathrm{Aut}\left(\F_{2}\right)$
on $\Sigma_{2,-2}\left(p\right)$.
\end{enumerate}
First, let us define $\Sigma_{2,-2}\left(p\right)$ properly. For
$A,B\in\mathrm{PSL}\left(2,p\right)$, define 
\begin{equation}
\Tr\left(A,B\right)\stackrel{\mathrm{def}}{=}\left(\tr A,\tr B,\tr AB\right)\in\nicefrac{\mathbb{F}_{p}^{~3}}{\sim},\label{eq:Tr for pairs of matrices}
\end{equation}
where $\sim$ is the equivalence of changing the sign of two of the
coordinates (each of $A$ and $B$ is a well-defined matrix in $\mathrm{SL}\left(2,p\right)$
up to a sign). Assume $\left\langle A,B\right\rangle =\mathrm{PSL}\left(2,p\right)$,
and let $\varphi\colon\F_{2}\twoheadrightarrow\mathrm{PSL}\left(2,p\right)$
be the epimorphism mapping the generators $a$ and $b$ of $\F_{2}$
to $A$ and $B$, respectively. The kernel $N=\ker\varphi$ is a $\mathrm{PSL}\left(2,p\right)$-defining
subgroup of $\F_{2}$, and define
\[
\Tr\left(N\right)\stackrel{\mathrm{def}}{=}\Tr\left(A,B\right).
\]
Recall that $\Sigma_{2}\left(G\right)$ denotes the set of $G$-defining
subgroups of $\F_{2}$.
\begin{claim}
The map $\Tr\colon\Sigma_{2}\left(\mathrm{PSL}\left(2,p\right)\right)\to\nicefrac{\mathbb{F}_{p}^{~3}}{\sim}$
is well-defined.
\end{claim}
\begin{proof}
Let $G=\mathrm{PSL}\left(2,p\right)$. Given $N\in\Sigma_{2}\left(G\right)$,
all epimorphisms $\F_{2}\twoheadrightarrow G$ with kernel $N$ are
obtained one from the other by post-composition with some automorphism
from $\mathrm{Aut}\left(G\right)$. But every automorphism of $G$
is obtained by a conjugation by some element from $\mathrm{PGL}\left(2,p\right)$.
Evidently, such conjugation does not effect the image of $\Tr$ on
the images of the generators $a$ and $b$ of $\F_{2}$.
\end{proof}
Recall that $\tr\left(\left[A,B\right]\right)=Q\left(\tr A,\tr B,\tr AB\right)$
where $Q\left(x,y,z\right)=x^{2}+y^{2}+z^{2}-xyz-2$. Thus, for $N\in\Sigma_{2}\left(\mathrm{PSL}\left(2,p\right)\right)$,
the element
\[
Q\left(N\right)\stackrel{\mathrm{def}}{=}Q\left(\Tr\left(N\right)\right)\in\mathbb{F}_{p}
\]
is well-defined, and we denote\marginpar{$\Sigma_{2,-2}\left(p\right)$}
\[
\Sigma_{2,-2}\left(p\right)\stackrel{\mathrm{def}}{=}Q^{-1}\left(-2\right)\subseteq\Sigma_{2}\left(\mathrm{PSL}\left(2,p\right)\right).
\]
Note that, by definition, for every $N\in\Sigma_{2,-2}\left(p\right)$
the triple $\Tr\left(N\right)$ is (an equivalence class up to sign
changes of) a solution to the Markoff equation (\ref{eq:markoff})
over $\nicefrac{\mathbb{Z}}{p\mathbb{Z}}$.
\begin{claim}
The map $\Tr\Big|_{\Sigma_{2,-2}\left(p\right)}$ is a bijection from
$\Sigma_{2,-2}\left(p\right)$ to $Y^{*}\left(p\right)$.
\end{claim}
\begin{proof}
Consider the map $\widetilde{\Tr}\colon\mathrm{SL}\left(2,p\right)\times\mathrm{SL}\left(2,p\right)\to\mathbb{F}_{p}^{~3}$
defined as in (\ref{eq:Tr for pairs of matrices}). By \cite[Theorems 2 and 3]{macbeath1969generators},
if $\left(x,y,z\right)\in\mathbb{F}_{p}^{~3}$ is the image of some
generating pair in $\mathrm{SL}\left(2,p\right)$, then every two
pairs in $\widetilde{\Tr}^{-1}\left(\left(x,y,z\right)\right)$ are
conjugated one to the other by an element $g\in\mathrm{SL}\left(2,\overline{\mathbb{F}_{p}}\right)$.
Since these pairs are generating, this conjugation by $g$ is an automorphism
of $\mathrm{SL}\left(2,p\right)$. As every automorphism of $\mathrm{SL}\left(2,p\right)$
is also an automorphism of $\mathrm{PSL}\left(2,p\right)$, we obtain
that 
\[
\Tr\Big|_{\Sigma_{2,-2}\left(p\right)}\colon\Sigma_{2,-2}\left(p\right)\to\nicefrac{\mathbb{F}_{p}^{~3}}{\sim}
\]
is injective.

By \cite[Thm 1]{macbeath1969generators}, the map $\widetilde{\Tr}$
is surjective. The analysis in \cite[Section 11]{mccullough2013nielsen}
shows that the only triple $\left(x,y,z\right)\in\mathbb{F}_{p}^{~3}$
with $Q\left(x,y,z\right)=-2$ which does not correspond to generating
pairs is\footnote{To see that $\left(0,0,0\right)$ is not associated with a generating
pair, note that if $A\in\mathrm{PSL}\left(2,p\right)$ has trace $0$,
then $A$ is an involution. If both $A$ and $B$ are involutions,
then $\left\langle A,B\right\rangle $ is a dihedral group, which
is a proper subgroup of $\mathrm{PSL}\left(2,p\right)$. } $\left(0,0,0\right)$. This completes the proof of the claim.
\end{proof}
We have left to show the isomorphism of $Q_{p}$ and the permutation
group induced by $\mathrm{Aut}\left(\F_{2}\right)$ on $Y^{*}\left(p\right)\cong\Sigma_{2,-2}\left(p\right)$.
Recall that $Q_{p}=\left\langle \tau_{\left(12\right)},\tau_{\left(23\right)},R_{3}\right\rangle $.
For $\F_{2}=\F\left(a,b\right)$, $\mathrm{Aut}\left(\F_{2}\right)$
is generated by the following Nielsen moves\footnote{We deliberately copy the notation for these Nielsen moves from \cite{mccullough2013nielsen}.}:
$r\colon\left(a,b\right)\mapsto\left(a^{-1},b\right)$, $s\colon\left(a,b\right)\mapsto\left(b,a\right)$
and $t\colon\left(a,b\right)\mapsto\left(a^{-1},ab\right)$. The induced
action of these three automorphisms on $Y^{*}\left(p\right)$ is easily
seen to be the same action given by $R_{3}$, $\tau_{\left(12\right)}$
and $\tau_{\left(23\right)}$, respectively. 

\section*{Appendix}

\begin{appendices}

\section{On the order of a quadratic integer modulo most primes\label{sec:Dan-Carmon}\protect \\
By Dan Carmon}

Throughout this appendix, we use the notation $f\ll g$ to mean that
there exists an absolute constant $C>0$ for which $f\le Cg$ for
all valid values of the implicit variables. The similar notation $f\ll_{a}g$
means there exists a function $C=C(a)>0$ for which $f\le Cg$. The
notation $f\asymp g$ is shorthand for ``$f\ll g$ and $g\ll f$''.

\subsection*{The main claim}

Let $a\in\mathbb{Q}(\sqrt{D})$ be a fixed quadratic integer with
norm 1 and absolute value $|a|>1$ (e.g. $a=\frac{3+\sqrt{5}}{2}$).
For primes $p\nmid D$, consider the residue $\bar{a}=(a\!\!\mod p)$,
as an element of either $\mathbb{F}_{p}$ or $\mathbb{F}_{p^{2}}$,
depending on whether $D$ is a quadratic residue modulo $p$. In both
cases there are two possible choices for $\bar{a}$, but its order
$o_{p}(a)$, which is the smallest positive integer satisfying $\bar{a}^{o_{p}(a)}=1\in\mathbb{F}_{p^{2}}$
is well-defined. Let $\pi(x)=\#\{p\le x:p\text{ -- prime}\}$\marginpar{$\pi\left(x\right)$}
be the prime counting function. We prove the following claim:
\begin{prop}
\label{prop:main appendix prop} For any constant $C\ge1$, 
\begin{equation}
\#\{p\le x:o_{p}(a)\le C\sqrt{x}\}\ll_{a}\frac{\pi(x)}{(\log x)^{\delta}(\log\log x)^{3/2-\delta}},
\end{equation}
where $\delta$ is the Erd\H{o}s-Tenenbaum-Ford constant, 
\[
\delta=1-\frac{1+\log\log2}{\log2}=0.086071....
\]
In particular, the set of primes with $o_{p}(a)>C\sqrt{p}$ has relative
density 1.
\end{prop}

\subsection*{Proof outline}

Proposition \ref{prop:main appendix prop} follows from the combination
of two sub-lemmas: 
\begin{lem}
\label{lem 1} Let $\alpha=\alpha(x)$ tend to infinity arbitrarily
slowly with $x$, and let $y=\sqrt{\frac{x}{\alpha}}$. Then 
\begin{equation}
\#\{p\le x:o_{p}(a)\le y\}\ll_{a}\frac{\pi(x)}{\alpha}.\label{lem 1 eq}
\end{equation}
\end{lem}

\begin{lem}
\label{lem 2} Let $\alpha,y$ be as in the previous Lemma. Define
$z=C\sqrt{x}$, and $u_{0}=\frac{\log\alpha}{\log x}$. Suppose further
that $\alpha\in\left(\frac{4}{C^{2}},\frac{\sqrt{x}}{C}\right)$.
Then 
\begin{equation}
\#\{p\le x:\exists d\in(y,z],\ p\equiv\pm1\!\!\!\pmod d\}\ll u_{0}^{\delta}\left(\log\tfrac{2}{u_{0}}\right)^{-3/2}\pi(x).\label{lem 2 eq}
\end{equation}
\end{lem}
Indeed, since $a$ has norm 1, $o_{p}(a)$ is always a factor of either
$p-1$ when $D$ is a quadratic residue modulo $p$, or of $p+1$
when $D$ is a non-quadratic residue, i.e. $p\equiv\pm1\!\pmod{o_{p}(a)}$
in either case. Thus $o_{p}(a)\le C\sqrt{x}$ implies that $p$ is
either included in the set of the first lemma if $o_{p}(a)\le y$,
or in the set of the second lemma if $o_{p}(a)\in(y,z]$. Choosing
the optimal value 
\begin{equation}
\alpha=(\log x)^{\delta}(\log\log x)^{3/2-\delta}
\end{equation}
yields the claimed value in the right hand side of both lemmas.

\subsection*{Proofs of the lemmas}
\begin{proof}[Proof of Lemma \ref{lem 1}]
 The following proof is an adaptation of an argument from Erd\H{o}s
and Murty \cite[Introduction]{EM}, in which only integral values
$a$ and a specific choice of $\alpha$ were considered.

For every $k\ge1$ define $A_{k}=\frac{a^{k}-a^{-k}}{\sqrt{D}}$.
Note that $A_{k}$ is always an integer, with $|A_{k}|<|a|^{k}$,
and that $o_{p}(a)=k$ implies $p\mid A_{k}$. Define 
\[
B_{y}=\prod_{k=1}^{\left\lfloor y\right\rfloor }A_{k},
\]
so that $o_{p}(a)\le y$ implies $p\mid B_{y}$. We now observe that
\begin{equation}
\log B_{y}=\sum_{k=1}^{\left\lfloor y\right\rfloor }\log A_{k}\le\sum_{k=1}^{\left\lfloor y\right\rfloor }k\log|a|\ll_{a}y^{2}=\frac{x}{\alpha},
\end{equation}
and on the other hand 
\begin{align}
\log B_{y} & \ge\sum_{p\mid B_{y}}\log p\ge\sum_{p\;:\;o_{p}(a)\le y}\log p\ge\sum_{\substack{\ensuremath{\sqrt{x}<p\le x}\\
o_{p}\left(a\right)\le y
}
}\log\sqrt{x}\\
 & =\tfrac{1}{2}\log x\cdot\#\{\sqrt{x}<p\le x:o_{p}(a)\le y\},\nonumber 
\end{align}
whence 
\begin{align}
\#\{p\le x:o_{p}(a)\le y\}\le\pi(\sqrt{x})+\frac{2\log B_{y}}{\log x}\ll_{a}\frac{2}{\alpha}\frac{x}{\log x}\ll\frac{\pi(x)}{\alpha}.
\end{align}
\end{proof}

\begin{proof}[Proof of Lemma \ref{lem 2}]
 This lemma is a direct application of results due to Ford \cite{ford2008distribution}.
We cite the relevant definitions and theorems. Ford's main object
of study is the function 
\[
H(x,y,z)=\#\{n\le x:\exists d\in(y,z],d\mid n\}.
\]
We are particularly interested in the specialized function 
\[
H(x,y,z;P_{\lambda})=\#\{n\le x:n\in P_{\lambda},\exists d\in(y,z],d\mid n\},
\]
where $P_{\lambda}=\{p+\lambda:p\text{ -- prime}\}$ is a set of shifted
primes, and more specifically only for $\lambda=\pm1$.

In \cite[Theorem 1]{ford2008distribution}, Ford estimates $H(x,y,z)$
for all possible choices of $y\le z\le x$. The relevant case for
our choice of $y,z$ is the third subcase of case (v), wherein $x,y,z$
are all large, $y\le\sqrt{x}$, and $z\in[2y,y^{2}]$, all of which
are immediately validated for our values, due to the constraint on
$\alpha$. For this case, the theorem states 
\begin{equation}
\frac{H(x,y,z)}{x}\asymp u^{\delta}\left(\log\tfrac{2}{u}\right)^{-3/2},\label{H est}
\end{equation}
where $u$ is the number satisfying $z=y^{1+u}$, or equivalently
\begin{align}
u=\frac{\log(z/y)}{\log y}=\frac{\log(C\sqrt{\alpha})}{\log(\sqrt{x/\alpha})}=\frac{\log\alpha+2\log C}{\log x-\log\alpha}\asymp\frac{\log\alpha}{\log x}=u_{0}.\label{u est}
\end{align}
In \cite[Theorem 6]{ford2008distribution}, Ford estimates $H(x,y,z;P_{\lambda})$,
for any fixed non-zero $\lambda$. The behaviour of the function is
determined by whether $z$ is greater or lesser than $y+(\log y)^{2/3}$.
The constraint on $\alpha$ implies $z\ge2y$, so we are certainly
in the regime of $z\ge y+(\log y)^{2/3}$, in which the theorem yields
\begin{equation}
H(x,y,z;P_{\lambda})\ll_{\lambda}\frac{H(x,y,z)}{\log x}.\label{lam est}
\end{equation}
Combining the estimates (\ref{H est}),(\ref{u est}),(\ref{lam est})
yields (\ref{lem 2 eq}), proving the lemma.
\end{proof}
\end{appendices}

\bibliographystyle{alpha}
\bibliography{markoff-refs}{}

\noindent Chen Meiri,\\
Department of Mathematics,\\
Technion - Israel Institute of Technology\\
Haifa 32000 Israel\\
chenm@tx.technion.ac.il\\

\noindent Doron Puder, \\
School of Mathematical Sciences,\\
Tel-Aviv University,\\
Tel-Aviv 69978 Israel\\
doronpuder@gmail.com\\

\noindent Dan Carmon, \\
School of Mathematical Sciences,\\
Tel-Aviv University,\\
Tel-Aviv 69978 Israel\\
dancarmo@post.tau.ac.il
\end{document}